\documentclass[a4paper,12pt]{article}
\usepackage{amssymb}
\usepackage[T1]{fontenc}
\usepackage{amsfonts}
\usepackage{amsmath, amsthm}
\newtheorem{theorem}{Theorem}

\newtheorem{definition}[theorem]{Definition}

\usepackage{graphicx,multirow}
\usepackage{float}

\usepackage{epstopdf}% To incorporate .eps illustrations using PDFLaTeX, etc.
\usepackage{subfigure}% Support for small, `sub' figures and tables

\begin{document}

\title{Mathematical model for pest-insect control using mating disruption and trapping}
\author{Roumen Anguelov$^{a}$, Claire Dufourd$^{a}$, Yves Dumont$^{b,}$\footnote{ Corresponding author: yves.dumont@cirad.fr}\\
$^a$Department of Mathematics and Applied Mathematics, University of Pretoria, \\ Pretoria, South Africa\\
$^b$CIRAD, Umr AMAP, Montpellier, France}

\maketitle

\begin{abstract}
Controlling pest insects is a challenge of main importance to preserve crop production. In the context of Integrated Pest Management (IPM) programs, we develop a generic model to study the impact of mating disruption control using an artificial female pheromone to confuse males and adversely affect their mating opportunities. Consequently the reproduction rate is diminished leading to a decline in the population size. For more efficient control, trapping is used to capture the males attracted to the artificial pheromone.  
The model, derived from biological and ecological assumptions, is governed by a system of ODEs.
A theoretical analysis of the model without control is first carried out to establish the properties of the endemic equilibrium. Then, control is added and the theoretical analysis of the model enables to identify threshold values of pheromone which are practically interesting for field applications. In particular, we show that there is a threshold above which the global asymptotic stability of the trivial equilibrium is ensured, i.e. the population goes to extinction. Finally we illustrate the theoretical results via numerical experiments.
\end{abstract}

\paragraph*{key words:} Ordinary differential equations; monotone dymanical system; stability analysis; pest control; pheromone traps; mating disruption; insect trapping
\section{Introduction}
\label{Introduction}

Pest insects are responsible for considerable damages on crops all over the world. Their presence can account for high production losses having repercussions on trading and exports as well on the sustainability of small farmers whose incomes entirely rely on their production. Exotic pests, can be particularly harmful as they can exhibit high invading potential due to the lack of natural enemies and their capacity to adapt to wide range of hosts and/or climate conditions. Therefore, pest management is essential to prevent devastating impact on economy, food security, social life, health and biodiversity. 

Chemical pesticides have long been used to control pest populations. However, their extensive use can have undesired side effects on the surrounding environment, such as reduction of the pest's natural enemies and pollution. Further, the development of insect resistance to the chemical lead to the need of using stronger and more toxic pesticides to maintain their efficacy. Thus, extensive use of pesticides is not a sustainable solution for pest control.    
Constant efforts are being made to reduce the toxicity of the pesticides for applicators and consumers, and alternative methods are being developed or improved to satisfy the charter of Integrated Pest Management (IPM) programs~\cite{apple1976integrated}. IPM aims to maintain pest population at low economic and epidemiological risk while respecting specific ecological and toxicological environmentally friendly requirements.

The Sterile Insect Technique (SIT), Mass Annihilation Technique (MAT) or mating disruption are examples of methods part of IPM strategies. SIT consists in releasing large numbers of sterilised males to compete with wild males for female insemination, reducing the number of viable offspring, while MAT consists in reducing the number of one or both sexes by trapping using a species-specific attractant. Pheromones or para-pheromones are often used to manipulate the behaviour of a specific species~\cite{witzgall2010sex, howse1998insect}. In this work, particular interest is given to mating disruption control which aims to obstruct the communication among sexual partners using lures to reduce the mating rate of the pest and thus lead to long-term reduction of the population~\cite{carde1990principles}. 

Mating disruption using pheromones has been widely studied to control moth pests~\cite{csiro1982mechanisms, carde1995control} on various types of crops. An early demonstration of the applicability of MAT has been shown for the eradication of \textit{Bactrocera doraslis} in the Okinawa Islands in 1984~\cite{koyama1984eradication}. More recently, the method has shown to be successful for the control of \textit{Tuta absoluta} on tomato crops in Italian greenhouses~\cite{cocco2013control}. Other successful cases are reported in~\cite{carde1995control}, such as for the control of the pink bollworm \textit{Pectinophora gossypiella} which attacks cotton, or the apple codling moth \textit{Cydia pomonella}. However, mating disruption has sometimes been a failure as for the control of the coffee leaf miner \textit{Leucoptera coffeela}~\cite{ambrogi2006efficacy} or for the control of the tomato pest \textit{Tuta absoluta} mentioned above in open field conditions~\cite{michereff2000initial} where mating disruption did not manage to reduce the pest population.
According to~\cite{ambrogi2006efficacy,michereff2000initial}, the failure of the method may be attributed to composition and dosage of the pheromone and/or to a high abundance of insects. 
For mating disruption success, understanding the attraction mechanisms of the pest to the pheromone is important, like the minimum level response, the distance of attraction or the formulation of the pheromone used. Environmental constrains are also crucial factors to take into account. These include the climate, the wind, the crop's foliage, etc. Further, the population size must be accounted for in order to design appropriate control strategies. Thus, planning efficient and cost effective control is a real challenge which can explain the failure of the experiments mentioned above. Mathematical modelling can be very helpful to get a better understanding on the dynamics of the pest population, and various control strategies can be studied to optimise the control. Here we combine mating disruption using female-sex pheromones lures to attract males away from females in order to reduce the mating opportunities adversely affecting the rate reproduction. For more efficient control, lures can be placed in traps to reduce the male population.
In 1955, Knipling proposed a numerical model to assess the effect of the release of sterile males on an insect population where the rate of fertilisation depends on the density of fertile males available for mating~\cite{knipling1955possibilities}. It is worth mentioning that in terms of modelling significant similarities can be found between MAT control and SIT control as the purpose of both methods is to affect the capacity of reproduction of the species. 
In the seminal works of Knipling et al.~\cite{knipling1966population,knipling1979basic}, several approaches for the suppression of insect populations among which MAT and SIT. Further, an overview on the mathematical models for SIT control which of relevance to mating disruption control can be found in~\cite[Chapter 2.5]{dyck2005sterile}.

In this paper we built a generic model for the control of a pest population using mating disruption and trapping to study the effort required to reduce the population size below harmful level. The model is derived using general knowledge or assumptions on insects' biology and ecology. We consider a compartmental approach based on the life cycle and mating behaviour to model the temporal dynamics of the population which is governed by a system of Ordinary Differential Equations (ODE).
A theoretical analysis of the model is carried out to discuss the efficiency of the control using pheromone traps depending on the strength of the lure and the trapping efficacy. In particular we study the properties of the equilibria using the pheromone as a bifurcation parameter. 
We identify two threshold values of practical importance. One corresponds to the minimum amount of lure required to affect the female population equilibrium, while the second one is the threshold above which extinction of the population is achieved. We also show that on small enough populations, at invasion stage for instance, extinction may be achieved with a small amount of lure.
We also show that combining mating disruption with trapping significantly reduces the amount of pheromone needed to obtain a full control of the population. 
The modelling approach for mating disruption control considered in~\cite{barclay1983pheromone,barclay1995models,fisher1985density} 
%is an extension of the works of Barclay and Van den Driessche~\cite{barclay1983pheromone} and of Fisher et al.~\cite{fisher1985density} 
where MAT control is modelled via  discrete density-dependent models. In the later, the authors identified a threshold value for the amount of pheromones above which the control of an insect population is possible. Similar qualitative results may be found in ~\cite{barclay1980sterile} in the case of SIT control, however not presented as in such depth as in the present study.

In the first section, we give a description of the model without control and analyse it theoretically. In the following section we describe the model with control. To model the impact of the mating disruption, we use a similar approach as the one proposed by Barclay and Van den Driessche~\cite{barclay1983pheromone}, where the amount of the artificial pheromone is given in terms of equivalent number of females. The model is studied theoretically, identifying threshold values which determine changes in the dynamics of the population. Finally, we perform numerical simulations to illustrate the theoretical results and we discuss their biological relevance.

% % % % % % % % % % % % % % % % % % % % % % % % % % % % % % % % % % % % % % % % % % % %
% The compartmental model for the dynamics of the insect
% % % % % % % % % % % % % % % % % % % % % % % % % % % % % % % % % % % % % % % % % % % %
\section{The compartmental model for the dynamics of the insect}
\label{SecNoControl}

We consider a generic model to describe the dynamics of a pest insect population based on biological and behavioural assumptions. For many pest species, such as fruit flies or moths, two main development stages can be considered: the immature stage, denoted $I$, which gathers eggs, larvae and pupae, and the adult stage. Typically, the adult female is the one responsible for causing direct damage to the host when laying her eggs. We split the adult females in two compartments, the females available for mating denoted $Y$, and the fertilised females denoted $F$. We assume that a mating female needs to mate with a male in order to pass into the compartment of the fertilised females and be able to deposit her eggs. Therefore, we also add a male compartment, denoted $M$, to study how the abundance of males impacts the transfer rate from $Y$ to $F$. We make the model sufficiently generic such that multiple mating can occur, which implies that fertilised females can become mating female again.

We denote $r$ the proportion of females emerging from the immature stage and entering the mating females compartment. Thus a proportion of $(1-r)$ on the immature enter the male compartment after emergence. We assume that the time needed for an egg laid to emergence is $1/\nu_I$, thus the transfer rate from $I$ to $Y$ or $M$ is $\nu_I$. Then, when males are in sufficient abundance to ensure the fertilisation of all the females available for mating, the transfer rate from $Y$ to $F$ is $\nu_Y$. However, if males are scarce, and if $\gamma$ is the number of females that can be fertilised by a single male, then only a proportion $\frac{\gamma M}{Y}$ of $Y$-females can pass into the $F$-females compartment. Therefore, the transfer rate from $Y$ to $F$ is modelled by the non-linear term $\nu_Y\min\{\frac{\gamma M}{Y},1\}$. Moreover, fertilised females go back to the mating females compartment with a rate of $\delta$. Further, the fertilised females supply the immature compartment with a rate $b\left(1-\frac{I}{K}\right)$, where $b$ is the intrinsic egg laying rate, while $K$ is the carrying capacity of the hosts. Finally, parameters $\mu_I$, $\mu_Y$, $\mu_F$ are respectively the death rates of compartments $I$, $Y$, $F$ and $M$. The flow diagram of the insects' dynamics is represented in Figure~\ref{Fig_LifeCycle}. 
\begin{figure}[H]
\begin{center}
\includegraphics[width=12cm]{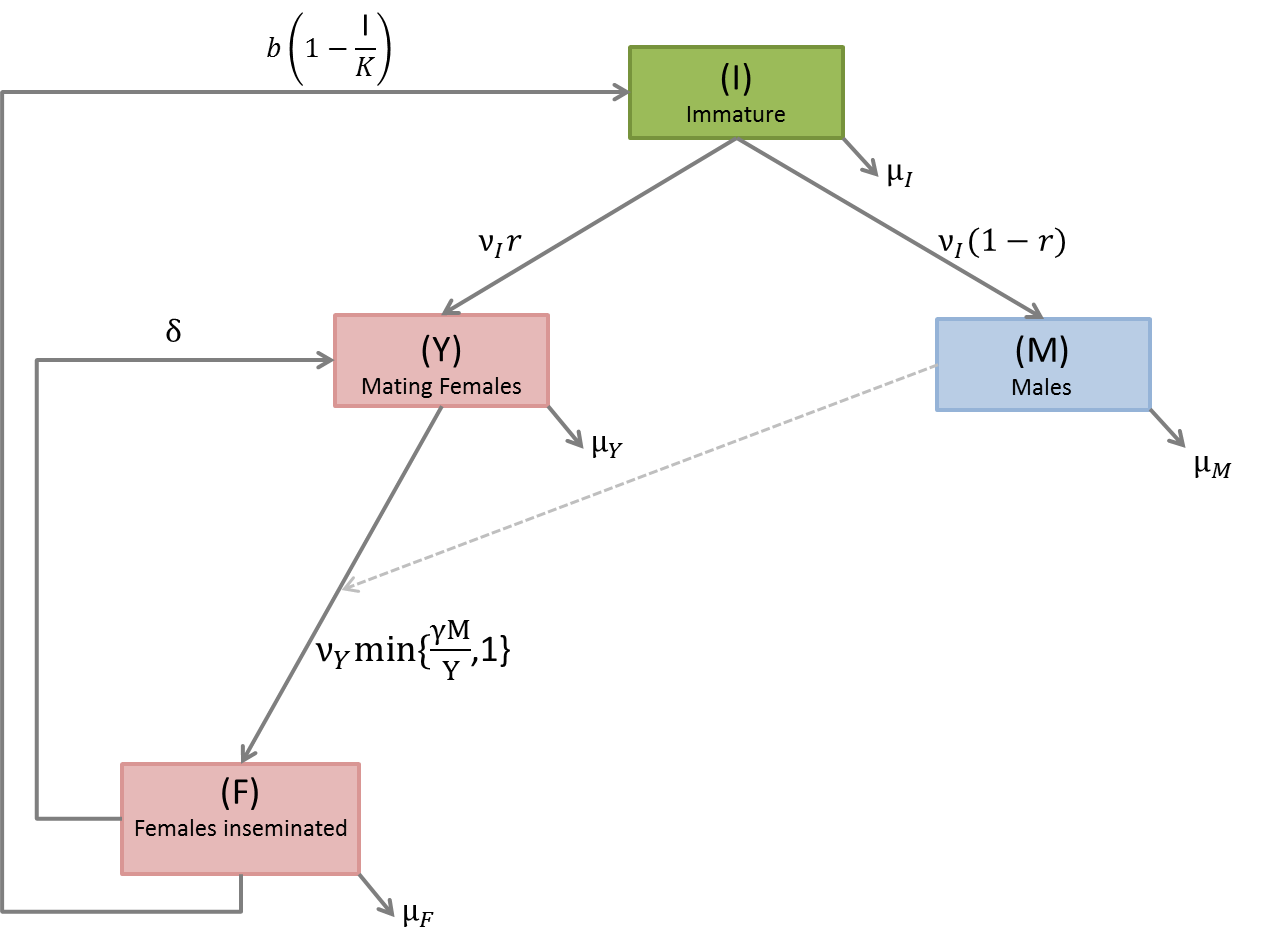}
\caption{Life cycle of the insect.}
\label{Fig_LifeCycle}
\end{center}
\end{figure}
The model is governed by the following system of ODEs:
\begin{equation}
\left\lbrace
	\begin{array}{lcl}
	\frac{dI}{dt} & = & b \left(1-\frac{I}{K} \right)F-\left( \nu_I+\mu_I \right)I, \\
	& &\\
	\frac{dY}{dt} & = & r\nu_I I -\nu_Y\min\{\frac{\gamma M}{Y},1\}Y+\delta F- \mu_{Y} Y, \\
	& &\\
	\frac{dF}{dt} & = & \nu_Y\min\{\frac{\gamma M}{Y},1\}Y -\delta F-\mu_{F}F, \\
	& &\\
	\frac{dM}{dt} & = & (1-r)\nu_I I - (\mu_M)M.
	\end{array}
	\label{EqModTemp0}
	\right.
\end{equation}
The list of the parameters used in the model are summarized in Table~\ref{Tab_ParaValues}. They are taken from~\cite{ekesi2006field} in order to perform numerical experiments.
\begin{table}[H]
\centering
{\begin{tabular}{lllll}
\hline
Parameter & Description & Unit & Value  \\
\hline
b & Intrinsic egg laying rate & female$^{-1}$day$^{-1}$ &  $9.272$\\
r & Female to male ratio & - & $0.57$\\
K & Carrying capacity & - & $1000$ \\
$\gamma$ & Females fertilised by a single male & - & $4$\\
$\mu_I$ & Mortality rate in the $I$ compartment   & day$^{-1}$ & $1/15$\\
$\mu_Y$ & Mortality rate in the $Y$ compartment& day$^{-1}$ & $1/75.1$\\
$\mu_F$ & Mortality rate in the $F$ compartment& day$^{-1}$ & $1/75.1$\\
$\mu_M$ & Mortality rate in the $M$ compartment&  day$^{-1}$ & $1/86.4$\\
$\nu_I$ & Transfer rate from $I$ to $Y$ &  day$^{-1}$ & $1/24.6$\\
$\nu_Y$ & Transfer rate from $Y$ to $F$ &  day$^{-1}$ & $0.5$\\
$\delta$ & Transfer rate from $F$ to $Y$ &  day$^{-1}$ & $0.1$\\
\hline
\end{tabular}}
\caption{List of parameters and there values used in the numerical simulations. }
\label{Tab_ParaValues}
\end{table}

% % % % % % % % % % % % % % % % % % % % % % % % % % % % % % % % % % % % % %
\subsection{Theoretical analysis of the model}
The theoretical analysis of the model is carried out for the case of male abundance and the case of male scarcity. These two cases are separated by the hyperplane $Y=\gamma M$. The analysis of the two systems can be carried out independently on the orthant $\mathbb{R}^4_+$.

% % % % % % % % % % % % % % % % % % % % % % % % % % % % % % % % % % % % % %
\subsubsection{Case 1: Male abundance}
In the case of male abundance, we recover the same model developed and studied in~\cite{dumont2011mathematical}. However, we follow another approach for the theoretical study. The system can be written in the vector form
 \begin{equation}
\frac{dx}{dt}=f(x),
\label{EqModAbundanceVector}
 \end{equation}
  with $x=(I,Y,F,M)^{T}$, and
 \begin{equation}
 f(x)=
 \left(
 \begin{array}{c}
 b\left(1-\frac{I}{K}\right)F-(\nu_I+\mu_I)I\\
 r\nu_I I +\delta F -(\nu_Y+\mu_Y)Y\\
 \nu_Y Y-(\delta+\mu_F)F\\
 (1-r)\nu_I I-\mu_M M
 \end{array}
 \right).
 \end{equation}
 Note that the right hand side of system (\ref{EqModAbundanceVector}), $f(x)$, is continuous and locally Lipschitz, so uniqueness and local existence of the solution are guaranteed. In proving the global properties we will use the fact that the system is cooperative on $\mathbb{R}^4_+$, that is the growth in any compartment impacts positively on the growth of all other compartments. For completeness of the exposition some basics of the theory of cooperative systems is given in Appendix C. Using Theorem \ref{TheoQuasiMon}, the system (\ref{EqModAbundanceVector}) is cooperative on $\Omega_K=\{x\in\mathbb{R}^4_+:I\leq K\}$  because
 the non-diagonal entries of the Jacobian matrix of $f$
 \begin{equation}
 \mathcal{J}_f=
 \left(
 \begin{array}{cccc}
 -\left( \nu_I+\mu_I +b\frac{F}{K} \right)& 0 & b(1-\frac{I}{K})  & 0 \\
 r\nu_I & -(\nu_Y+\mu_Y) & \delta & 0 \\
 0 & \nu_Y & -(\mu_F+\delta) & 0 \\
 (1-r)\nu_I& 0 & 0 & -\mu_M \\
 \end{array}
 \right),
 \end{equation}
are non-negative on $\Omega_K$.

The persistence of a population is typically linked to its Basic Offspring Number. For simple population models the basic offspring number is defined as the number of offspring produced by a single individual in their life time provided abundant resource is available. In general, the definition could be more complicated and its value is computed by using the next generation method. For the model in (\ref{EqModAbundanceVector}) the Basic Offspring Number is
\begin{equation}
\mathcal{N}_{0}=\frac{b r\nu_I \nu_Y}{(\mu_I+\nu_I)((\nu_Y+\mu_Y)(\delta+\mu_F)-\delta \nu_Y)}.
\label{NoMAT_BON}
\end{equation}
In  this paper we will use $\mathcal{N}_{0}$ only as a threshold parameter. The persistence when $\mathcal{N}_{0}>1$ and the extinction when $\mathcal{N}_{0}<1$ are given direct proofs. Hence, we will not discuss the details on specific properties of the number. For completeness of the exposition the computation of (\ref{NoMAT_BON}) is given in Appendix A.

\begin{theorem}
\begin{itemize}
\item[a)] The system of ODE (\ref{EqModAbundanceVector}) defines a positive dynamical system on $\mathbb{R}^4_+$.

\item[b)] If $\mathcal{N}_{0}\leq 1$ then $TE=(0,0,0,0)^T$ is a globally asymptotically stable (GAS) equilibrium.

\item[c)] If $\mathcal{N}_{0}> 1$ then $TE$ is an unstable equilibrium and the system admits a positive equilibrium $EE^*=(I^*,Y^*,F^*,M^*)^T$, where
\begin{eqnarray}
I^* & = & \left( 1-\frac{1}{\mathcal{N}_{0}} \right)K, \nonumber\\
Y^* & = & \frac{r \nu_I (\delta+\mu_F)}{(\nu_Y+\mu_Y)(\delta+\mu_F)-\delta \nu_Y} \left( 1-\frac{1}{\mathcal{N}_{0}} \right)K, \nonumber \\
F^* & = & \frac{r \nu_I \nu_Y}{(\nu_Y+\mu_Y)(\delta+\mu_F)-\delta \nu_Y} \left( 1-\frac{1}{\mathcal{N}_{0}} \right)K, \nonumber \\
M^* & = & \frac{(1-r)\nu_I}{\mu_M} \left( 1-\frac{1}{\mathcal{N}_{0}} \right)K, \nonumber
\end{eqnarray}
which is a globally asymptotically stable (GAS) on $\mathbb{R}^4_+\setminus \{x\in \mathbb{R}^4_+:I=Y=F=0\}$.
\end{itemize}
\label{Prop_PosiDynSysdelta}
\end{theorem}

\begin{proof}
a) Let $q\in \mathbb{R}$, $q\geq K$. Denote
\begin{equation}
y_q=
\left(
\begin{array}{c}
K\\
\frac{r\nu_I(\delta+\mu_F) }{(\nu_Y+\mu_Y)(\delta+\mu_F)-\delta \nu_Y}q\\
\frac{r\nu_I \nu_Y }{(\nu_Y+\mu_Y)(\delta+\mu_F)-\delta \nu_Y}q\\
\frac{(1-r)\nu_I }{\mu_M}q\\
\end{array}
\right)
\end{equation}
We have $f(\textbf{0})=\textbf{0}$ and $f(y_q)\leq \textbf{0}$. Then, from Theorem~\ref{Thm_GAS_MonSys} it follows that (\ref{EqModAbundanceVector}) defines a positive dynamical system on $[\textbf{0},y_q]$. That is $\forall x\in [\textbf{0},y_q]$, the problem (\ref{EqModAbundanceVector}), given an initial condition $x(0)=x_0\in [\textbf{0},y_q]$, admits a unique solution for all $t\in[0,\infty]$ in $[\textbf{0},y_q]$~\cite{anguelov2010topological}.
Hence (\ref{EqModAbundanceVector})  defines a positive dynamical system on $\Omega_K=\cup_{q \geq K}[\textbf{0},y_q]$. Let now $x_0\in\mathbb{R}^4_+\setminus\Omega_K$. Using that the vector field defined by $f$ points inwards on the boundary of $\mathbb{R}^4_+$ we deduce that the solution initiated at $x_0$ remains in $\mathbb{R}^4_+$. Then one can see from the first equation of (\ref{EqModAbundanceVector}) that $I(t)$ decreases until the solution is absorbed in $\Omega_K$. Then (\ref{EqModAbundanceVector}) defines a dynamical system on $\mathbb{R}^4_+$ with $\Omega_K$ being an absorbing set.

b)Solving $f(x)=0$ yields two solutions $TE$ and $EE^*$. When $\mathcal{N}_0\leq 1$ $TE$ is the only equilibrium in $\mathbb{R}^4_+$. Then for any $q\geq K$ $TE$ is the only equilibrium in $[\textbf{0},y_q]$. It follows from Theorem~\ref{Thm_GAS_MonSys} that it is GAS on $[\textbf{0},y_q]$. Therefore, $TE$ is GAS on $\Omega_K$ and further on $\mathbb{R}_+^{4}$ by using that $\Omega_K$ is an absorption set.

c) For $\varepsilon\in (0,I^*)$ we consider the vector
\begin{equation}
z_{\varepsilon}=
\left(
\begin{array}{c}
\varepsilon\\
\frac{r\nu_I(\delta+\mu_F) }{(\nu_Y+\mu_Y)(\delta+\mu_F)-\delta \nu_Y}\varepsilon\\
\frac{r\nu_I \nu_Y }{(\nu_Y+\mu_Y)(\delta+\mu_F)-\delta \nu_Y}\varepsilon\\
\frac{(1-r)\nu_I }{\mu_M}\varepsilon\\
\end{array}
\right).
\end{equation}
Straightforward computations show that $\frac{dY}{dt}(z_{\varepsilon})=\frac{dF}{dt}(z_{\varepsilon})=\frac{dM}{dt}(z_{\varepsilon})=0$, and that we have
\begin{equation}
\begin{array}{lcl}
\frac{dI}{dt}(z_{\varepsilon})&=& b \left(1-\frac{\varepsilon}{K} \right)\frac{r\nu_I \nu_Y }{(\nu_Y+\mu_Y)(\delta+\mu_F)-\delta \nu_Y}\varepsilon-\left( \nu_I+\mu_I \right)\varepsilon, \\
&=& \frac{b r\nu_I \nu_Y }{(\nu_Y+\mu_Y)(\delta+\mu_F)-\delta \nu_Y}\varepsilon-\frac{b r\nu_I \nu_Y }{(\nu_Y+\mu_Y)(\delta+\mu_F)-\delta \nu_Y}\frac{\varepsilon^2}{K}-\left( \nu_I+\mu_I \right)\varepsilon\\
&=& (\nu_I+\mu_I)\mathcal{N}_{0}\varepsilon-(\nu_I+\mu_I)\mathcal{N}_{0}\frac{\varepsilon^2}{K}-\left( \nu_I+\mu_I \right)\varepsilon\\
&=& (\nu_I+\mu_I)\left(\mathcal{N}_{0}(1-\frac{\varepsilon}{K})-1 \right)\varepsilon\\
&=& (\nu_I+\mu_I)\frac{\mathcal{N}_0}{K}(I^*-\varepsilon)\varepsilon>0.
\end{array}
\nonumber
\end{equation}
Thus, $f(x_{\varepsilon})\geq\textbf{0}$. For any $\varepsilon\in(0,I^*)$ and $q\in (\varepsilon, K]$ $EE^*$ is a unique equilibrium in the interval $[z_\varepsilon,y_q]$ and hence, by Theorem \ref{Thm_GAS_MonSys}, it is GAS on this interval. Therefore, $EE^*$ is GAS also on $interior(\Omega_K)$. Using further that (i) if either $I(0)$ or $Y(0)$ or $F(0)$ is positive then $x(t)>0$ for all $t>0$ and that (ii) $\Omega_K$ is an absorbing set, we obtain that $EE^*$ attracts all solution in $\mathbb{R}^4_+$  except the ones initiated on the $M$-axis.
\end{proof}

% % % % % % % % % % % % % % % % % % % % % % % % % % % % % % % % % % % % % %
\subsubsection{Case 2: Male scarcity}
Next we consider the case when males are scarce, that is when $\gamma M < Y$. The system assumes the form
\begin{equation}
\left\lbrace
	\begin{array}{lcl}
	\frac{dI}{dt} & = & b \left(1-\frac{I}{K} \right)F-\left( \nu_I+\mu_I \right)I, \\
	& &\\
	\frac{dY}{dt} & = & r\nu_I I -\nu_Y\gamma M+\delta F- \mu_{Y} Y, \\
	& &\\
	\frac{dF}{dt} & = & \nu_Y\gamma M -(\delta +\mu_{F})F, \\
	& &\\
	\frac{dM}{dt} & = & (1-r)\nu_I I - \mu_M M.
	\end{array}
	\right.
	\label{EqModTempScarceNoMAT}
\end{equation}
It is easy to see that the second equation can be decoupled. The system of the remaining $3$ equations is of the form
\begin{equation}
	\frac{du}{dt}=g(u),
	\label{EqModTempScarceNoMAT_NoY}
\end{equation}
where
$
u=\left(
\begin{array}{c}
I\\
F\\
M\\
\end{array}
\right)
$
and
$
g(u)=\left(
\begin{array}{c}
b \left(1-\frac{I}{K} \right)F-\left( \nu_I+\mu_I \right)I\\
\nu_Y\gamma M -(\delta +\mu_{F})F\\
(1-r)\nu_I I - \mu_M M\\
\end{array}
\right)
$.
The non-diagonal entries of the Jacobian of $g$,
$$
\mathcal{J}_g=\left(
\begin{array}{ccc}
-(\nu_I+\mu_I + b\frac{F}{K}) & b \left(1-\frac{I}{K} \right) & 0 \\
0 & -(\delta+\mu_F) & \nu_Y \gamma\\
(1-r)\nu_I & 0 & -\mu_M\\
\end{array}
\right),
$$
are non-negative. Hence the system (\ref{EqModTempScarceNoMAT_NoY}) is cooperative.
The Basic Offspring Number for system (\ref{EqModTempScarceNoMAT_NoY}) is
\begin{equation}
\hat{\mathcal{N}}_0=\frac{b\gamma(1-r)\nu_I\nu_Y}{(\nu_I+\mu_I)(\delta+\mu_F)\mu_M}.
\end{equation}

The following theorem describes the properties of system (\ref{EqModTempScarceNoMAT_NoY}).

\begin{theorem}
\begin{itemize}
\item[a)] The system of ODE (\ref{EqModTempScarceNoMAT_NoY}) defines a positive dynamical system on $\mathbb{R}^3_+$.

\item[b)] If $\hat{\mathcal{N}}_{0}\leq 1$ then $TE=(0,0,0)^T$ is a globally asymptotically stable (GAS) equilibrium.

\item[c)] If $\hat{\mathcal{N}}_{0}> 1$ then $TE$ is an unstable equilibrium and the system admits a positive equilibrium $\hat{EE}=(\hat{I},\hat{F},\hat{M})$, where
\begin{eqnarray}
\hat{I} & = & \left( 1-\frac{1}{\hat{\mathcal{N}}_{0}} \right)K, \nonumber\\
%\hat{Y} & = & \frac{r \nu_I (\mu_F+\delta)}{(\nu_Y+\mu_Y)(\mu_F+\delta)-\delta \nu_Y} \left( 1-\frac{1}{\mathcal{N}_{0}} \right)K, \nonumber \\
\hat{F} & = & \frac{\gamma(1-r) \nu_I \nu_Y}{(\delta+\mu_F)\mu_M} \left( 1-\frac{1}{\hat{\mathcal{N}}_{0}} \right)K, \nonumber \\
\hat{M} & = & \frac{(1-r)\nu_I}{\mu_M} \left( 1-\frac{1}{\hat{\mathcal{N}}_{0}} \right)K, \nonumber
\end{eqnarray}
which is a globally asymptotically stable (GAS) on $\mathbb{R}^3_+\setminus \{x\in \mathbb{R}^3_+:I=F=0\}$.
\end{itemize}
\label{Thm_ProperEquiNoMATScarce}
\end{theorem}

\begin{proof}
a) Let $q\in \mathbb{R}$, $q\geq K$. Denote
\begin{equation}
\hat{y}_q=
\left(
\begin{array}{c}
K\\
\frac{\gamma(1-r)\nu_I \nu_Y }{(\delta+\mu_F)\mu_M}q\\
\frac{(1-r)\nu_I }{\mu_M}q\\
\end{array}
\right)
\end{equation}
We have $g(\textbf{0})=\textbf{0}$ and $g(\hat{y}_q)\leq \textbf{0}$. Then, from Theorem~\ref{Thm_GAS_MonSys} it follows that (\ref{EqModTempScarceNoMAT_NoY}) defines a positive dynamical system on $[\textbf{0},\hat{y}_q]$.
Hence it defines a positive dynamical system on $\hat{\Omega}_K=\cup_{q \geq K}[\textbf{0},\hat{y}_q]$. Let now $x_0\in\mathbb{R}^3_+\setminus\hat{\Omega}_K$. Using that the vector field defined by $g$ points inwards on the boundary of $\mathbb{R}^3_+$ we deduce that the solution initiated at $x_0$ remains in $\mathbb{R}^3_+$. Then one can see from the first equation of (\ref{EqModTempScarceNoMAT_NoY}) that $I(t)$ decreases until the solution is absorbed in $\hat{\Omega}_K$. Then (\ref{EqModTempScarceNoMAT_NoY}) defines a dynamical system on $\mathbb{R}^3_+$ with $\hat{\Omega}_K$ being an absorbing set.

b)Solving $g(x)=0$ yields two solutions $TE$ and $\hat{EE}$. When $\hat{\mathcal{N}}_0\leq 1$, $TE$ is the only equilibrium in $\mathbb{R}^3_+$. Then for any $q\geq K$ $TE$ is the only equilibrium in $[\textbf{0},\hat{y}_q]$. It follows from Theorem~\ref{Thm_GAS_MonSys} that it is GAS on $[\textbf{0},\hat{y}_q]$. Therefore, $TE$ is GAS on $\hat{\Omega}_K$ and further on $\mathbb{R}_+^{3}$ by using that $\hat{\Omega}_K$ is an absorbing set.

c) For $\varepsilon\in (0,\hat{I})$ we consider the vector
\begin{equation}
\hat{z}_{\varepsilon}=
\left(
\begin{array}{c}
\varepsilon\\
\frac{\gamma(1-r)\nu_I \nu_Y }{(\delta+\mu_F)\mu_M}\varepsilon\\
\frac{(1-r)\nu_I }{\mu_M}\varepsilon\\
\end{array}
\right).
\end{equation}

Straightforward computations show that $\frac{dF}{dt}(\hat{z}_{\varepsilon})=\frac{dM}{dt}(\hat{z}_{\varepsilon})=0$, and that we have
\begin{equation}
\begin{array}{lcl}
\frac{dI}{dt}(\hat{z}_{\varepsilon})&=& b \left(1-\frac{\varepsilon}{K} \right)\frac{\gamma(1-r)\nu_I \nu_Y }{(\delta+\mu_F)\mu_M}\varepsilon-\left( \nu_I+\mu_I \right)\varepsilon \\
&=& \frac{b \gamma (1-r) \nu_I \nu_Y }{(\delta+\mu_F)\mu_M}\varepsilon-\frac{b \gamma (1-r)\nu_I \nu_Y }{(\delta+\mu_F)\mu_M}\frac{\varepsilon^2}{K}-\left( \nu_I+\mu_I \right)\varepsilon\\
&=& (\nu_I+\mu_I)\hat{\mathcal{N}}_{0}\varepsilon-(\nu_I+\mu_I)\hat{\mathcal{N}}_{0}\frac{\varepsilon^2}{K}-\left( \nu_I+\mu_I \right)\varepsilon\\
&=& (\nu_I+\mu_I)\left(\hat{\mathcal{N}}_{0}(1-\frac{\varepsilon}{K})-1 \right)\varepsilon\\
&=& (\nu_I+\mu_I)\frac{\hat{\mathcal{N}}_0}{K}(\hat{I}-\varepsilon)\varepsilon>0.
\end{array}
\nonumber
\end{equation}
Thus, $g(x_{\varepsilon})\geq\textbf{0}$. For any $\varepsilon\in(0,\hat{I})$ and $q\in (\varepsilon, K]$ $\hat{EE}$ is a unique equilibrium in the interval $[z_\varepsilon,y_q]$ and hence, by Theorem \ref{Thm_GAS_MonSys}, it is GAS on this interval. Therefore, $\hat{EE}$ is GAS also on $interior(\hat{\Omega}_K)$. Using further that (i) if either $I(0)$ or $F(0)$ is positive then $x(t)>0$ for all $t>0$ and that (ii) $\hat{\Omega}_K$ is an absorbing set, we obtain that $\hat{EE}$ attracts all solution in $\mathbb{R}^3_+$  except the ones initiated on the $M$-axis.
\end{proof}

It can be deduced from Theorem \ref{Thm_ProperEquiNoMATScarce} that the non-trivial equilibrium value of $Y$ is
\begin{equation}
\hat{Y}=\frac{r\nu_I(\delta+\mu_F)\mu_M-\nu_Y\gamma (1-r)\nu_I\mu_F}{\mu_Y(\delta+\mu_F)\mu_M}\left( 1-\frac{1}{\hat{\mathcal{N}}_{0}} \right)K.
\end{equation}
We note that the value of $\hat{Y}$ and in general the value of the variable $Y$ can be negative. Hence, the system of ODE (\ref{EqModTempScarceNoMAT}) defines a dynamical system on $\mathbb{R}_+\times\mathbb{R}\times\mathbb{R}^2_+$.

% % % % % % % % % % % % % % % % % % % % % % % % % % % % % % % % % % % % % %
\subsubsection{Conclusions for model (\ref{EqModTemp0})}

In what follows we assume that the population has an endemic equilibrium. Otherwise, no control would be necessary. Further, it is natural to assume that, at equilibrium, there is abundance of males. In terms of the parameters of the model these assumptions can be written as:
%\begin{subequations}\label{Assumption}
%\begin{equation}
%\mathcal{N}_0>1,
%\end{equation}
%\begin{equation}
%Y^*<\gamma M^*
%\end{equation}
%\end{subequations}
\begin{eqnarray}
& 1. & \mathcal{N}_0>1, \label{Assumption1}\\
& 2. & Y^*<\gamma M^*. \label{Assumption2}
\end{eqnarray}
Under the assumptions (\ref{Assumption1}) and (\ref{Assumption2}) we have that  $\hat{EE}>0$ and $\hat{Y}<\gamma \hat{M}$. Indeed, when $\mathcal{N}_0>1$, the inequality
$$\frac{\hat{\mathcal{N}}_0}{\mathcal{N}_0}>1$$
is equivalent to $$\gamma>\frac{r(\delta+\mu_F)\mu_M}{(1-r)((\nu_Y+\mu_Y)(\delta+\mu_F)-\delta\nu_Y)},$$
which is ensured by (\ref{Assumption2}). Therefore under assumptions (\ref{Assumption1}) and (\ref{Assumption2}) we have that $\hat{\mathcal{N}}_0>1$, and by Theorem \ref{Thm_ProperEquiNoMATScarce}, we have that the system (\ref{EqModTempScarceNoMAT}) has a non-trivial equilibrium, which is globally asymptotically stable on $\mathbb{R}_+\times\mathbb{R}\times\mathbb{R}^2_+\setminus\{(I,Y,F,M)^T:I=F=0\}$.

Further,  under (\ref{Assumption2}), we have
\begin{eqnarray}
\hat{Y}-\gamma\hat{M} & = & \hat{I}\left(\frac{r\nu_I(\delta+\mu_F)\mu_M-\nu_Y\gamma(1-r)\nu_I\mu_F}{\mu_Y(\delta+\mu_F)\mu_M}-\gamma\frac{(1-r)\nu_I}{\mu_M}\right) \nonumber\\
& = &\frac{\nu_I \hat{I}}{\mu_M}\left(\frac{r \mu_M}{\mu_Y}-\gamma(1-r)\frac{\nu_Y\mu_F+\mu_Y(\delta+\mu_F)}{\mu_Y(\mu_F+\delta)}\right) \nonumber \\
& < &\frac{\nu_I \hat{I}}{\mu_M}\left(\frac{r \mu_M}{\mu_Y}-\frac{r(\delta+\mu_F)\mu_M}{((\nu_Y+\mu_Y)(\delta+\mu_F)-\delta\nu_Y)}\frac{\nu_Y\mu_F+\mu_Y(\delta+\mu_F)}{\mu_Y(\delta+\mu_F)}\right) \nonumber \\
& < &\frac{\nu_I \hat{I}}{\mu_M}\left(\frac{r \mu_M}{\mu_Y}-\frac{r \mu_M}{\mu_Y}\right)=0. \nonumber\\
\end{eqnarray}

To summarize, under assumptions (\ref{Assumption1}) and (\ref{Assumption2}), the globally asymptotically stable equilibria of both (\ref{EqModAbundanceVector}) and  (\ref{EqModTempScarceNoMAT}) are in the male abundance region defined via $Y<\gamma M$. Hence the only non-trivial equilibrium of (\ref{EqModTemp0}) is $EE^*=(I^*,Y^*,F^*,M^*)^T$. In this way we obtain the following result:
\begin{theorem}\label{theoSummaryNoMat}
Given (\ref{Assumption1}) and (\ref{Assumption2}), the model (\ref{EqModTemp0}) has two equilibria:
\begin{itemize}
\item[a)] $TE$ which is unstable, and
\item[b)] $EE^*$ which is asymptotically stable.
\end{itemize}
\end{theorem}
The equilibrium $EE*$ attracts 
solutions which are entirely in the male abundance region, excluding the $M$-axis. Solutions in the male scarcity region are attracted to $\hat{EE}=(\hat{I},\hat{Y},\hat{F},\hat{M})^T$. Hence they leave the male scarcity region and enter the male abundance region. In the male abundance region they are governed by (\ref{EqModAbundanceVector}) and, therefore, attracted to $EE^*$. In general this reasoning does not exclude the possibility that a solution may leave the male abundance region, enter the male scarcity region and then leave it. However, our numerical simulations show that this happens only finite number of times, indicating that EE* is globally asymptotically stable on $\mathbb{R}_+^4\setminus\{x\in\mathbb{R}_+^4:I=F=0\}$. Hence, eventually such solution stays in the male abundance region and therefore converges to $EE^*$.

% \begin{theorem}\label{theoSummaryNoMat}
% Given (\ref{Assumption1}) and (\ref{Assumption2}), the model (\ref{EqModTemp0}) has two equilibria:
% \begin{itemize}
% \item[a)] $TE$ which is unstable, and
% \item[b)] $EE^*$ which is globally asymptotically stable on $\mathbb{R}_+^4\setminus\{x\in\mathbb{R}_+^4:I=F=0\}$.
% \end{itemize}
% \end{theorem}

% % % % % % % % % % % % % % % % % % % % % % % % % % % % % % % % % % % % % % % % % % % %
% The compartmental model for the dynamics of the insect
% % % % % % % % % % % % % % % % % % % % % % % % % % % % % % % % % % % % % % % % % % % %
\section{Modelling mating disruption and trapping}
\label{SecWithControl}

% % % % % % % % % % % % % % % % % % % % % % % % % % % % % % % % % % % % % %
\subsection{Mating disruption and trapping}
In order to maintain the pest population to a low level, we consider a control using female-pheromone-traps to disrupt male mating behaviour. More precisely, we take into account two aspects for the control. The first aspect consists of disturbing the mating between males and females to reduce the fertilisation opportunities, which in turn, reduces the number of offspring. This is done using traps that are releasing a female pheromone lure to which males are attracted. This leads to a reduction in the number of males available for mating near the females, and decreases the opportunity for fertilisation. The efficiency of mating disruption depends on the strength of the lure or on the number of traps in an area. The second aspect of the control is the trapping potential of the trap. We assume that the lure traps also contain an insecticide which can kill the captured insects.

% % % % % % % % % % % % % % % % % % % % % % % % % % % % % % % % % % % % % %
\subsection{The model}
In order to take in account the effect of the lures, we consider the approach proposed by Barclay and Van den Driessche~\cite{barclay1983pheromone,barclay2014models}. That is, the strength of the lure is represented as the quantity of pheromones released by an equivalent number of wild females. Thus, in the model the effect of the lure corresponds to the attraction of $Y_P$ additional females. In such a setting, the total number of ``females'' attracting males is $Y+Y_P$~\cite{barclay1983pheromone}. In particular, this means that males have a probability of $\displaystyle\frac{Y}{Y+Y_P}$ to be attracted to wild females, and a probability of $\displaystyle\frac{Y_P}{Y+Y_P}$ to be attracted to the pheromone traps.
Denote $\gamma$ the number of females that can be inseminated by a single male. Then, the transfer rate from $Y$ to $F$ does not exceed $\nu_Y\displaystyle\frac{\gamma M}{Y+Y_P}$. When $\displaystyle\frac{\gamma M}{Y+Y_P}>1$ the population is in a male abundance state and the transfer rate is $\nu_Y$. However, when $\displaystyle\frac{\gamma M}{Y+Y_P}<1$, then the population is in a male scarcity state and the transfer rate is $\nu_Y\displaystyle\frac{\gamma M}{Y+Y_P}$. Altogether, the transfer rate is $\nu_Y\min\{\frac{\gamma M}{Y+Y_P},1\}$.

The parameter $\alpha$ represents the death or capture rate for the fraction $\displaystyle\frac{Y_P}{Y+Y_P}$ of the males which are attracted by the lures. The flow diagram is represented in Figure \ref{fig_MATmodel} which yields the following system of ODEs:
\begin{equation}
\left\lbrace
	\begin{array}{lcl}
	\frac{dI}{dt} & = & b \left(1-\frac{I}{K} \right)F-\left( \nu_I+\mu_I \right)I, \\
	& &\\
	\frac{dY}{dt} & = & r\nu_I I -\nu_Y\min\{\frac{\gamma M}{Y+Y_P},1\}Y+\delta F- \mu_{Y} Y, \\
	& &\\
	\frac{dF}{dt} & = & \nu_Y\min\{\frac{\gamma M}{Y+Y_P},1\}Y -\delta F-\mu_{F}F, \\
	& &\\
	\frac{dM}{dt} & = & (1-r)\nu_I I - (\mu_M+\alpha\frac{Y_P}{Y+Y_P})M.
	\end{array}
	\label{EqModTemp}
	\right.
\end{equation}

\begin{figure}[H]
\centering
\includegraphics[width=12cm]{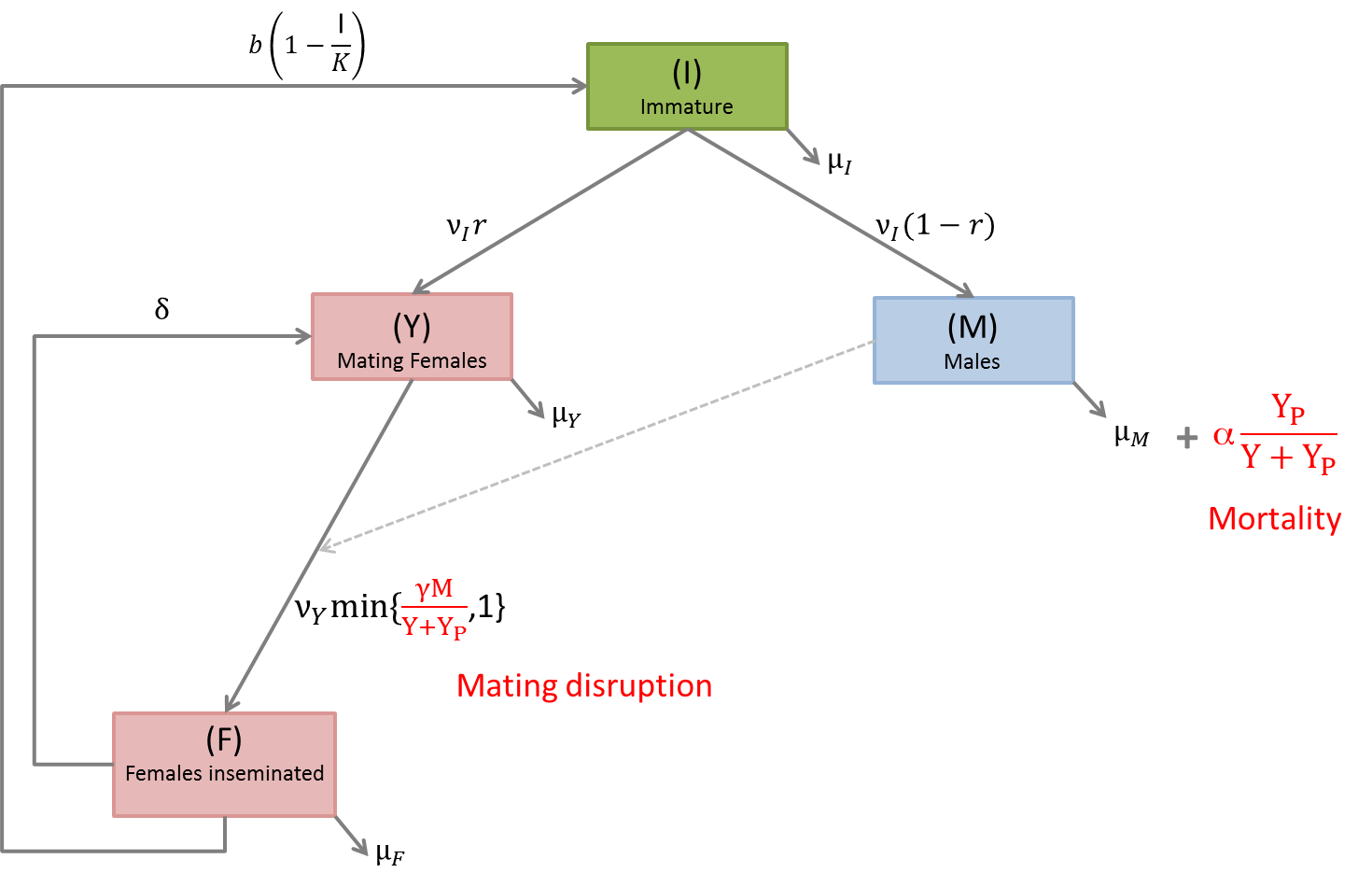}
\caption{Control model using mating disruption and trapping.}
\label{fig_MATmodel}
\end{figure}

% % % % % % % % % % % % % % % % % % % % % % % % % % % % % % % % % % % % %
\subsection{Theoretical analysis of the control model}
The aim of this section is to investigate the existence of equilibria of model (\ref{EqModTemp}) and their asymptotic properties. First we consider the model in the male abundance state and in the male scarcity state independently. Then we draw conclusions for the general model.

\subsubsection{Properties of the equilibria in the male abundance region ($\gamma M>Y+Y_P$)}\label{SubsubMaleAbund}

The dynamics of the population in the male abundance region $Y+Y_P<\gamma M$ are governed by the following system of ODEs:
\begin{equation}
\left\lbrace
	\begin{array}{lcl}
	\frac{dI}{dt} & = & b \left(1-\frac{I}{K} \right)F-\left( \nu_I+\mu_I \right)I, \\
	& &\\
	\frac{dY}{dt} & = & r\nu_I I -\nu_YY+\delta F- \mu_{Y} Y, \\
	& &\\
	\frac{dF}{dt} & = & \nu_YY -(\delta +\mu_{F})F, \\
	& &\\
	\frac{dM}{dt} & = & (1-r)\nu_I I - (\mu_M+\alpha\frac{Y_P}{Y+Y_P})M.
	\end{array}
	\label{EqModTempAbundance}
		\right.
\end{equation}
We note that the first three equations in (\ref{EqModTempAbundance}) as in (\ref{EqModAbundanceVector}) are the same, while the fourth equation in both systems can be decoupled. Then, using exactly the same method as in Theorem \ref{Prop_PosiDynSysdelta} we obtain the following theorem.

\begin{theorem}\begin{itemize}\item[a)] The system of ODEs (\ref{EqModTempAbundance}) defines a positive dynamical system on $\mathbb{R}_+^4$.
\item[b)] Under assumptions (\ref{Assumption1}) and (\ref{Assumption2}), the system has a positive equilibrium $EE^\#=(I^*,Y^*,F^*,M^\#(Y_P))^T$, where
\[
M^{\#}(Y_P)=\frac{(1-r)\nu_I (Y^*+Y_P)}{\mu_M(Y^*+Y_P)+\alpha Y_P}I^*=\frac{M^*}{1+\frac{\alpha Y_P}{\mu_M(Y^*+Y_P)}},
\]
which is globally asymptotically stable on $\mathbb{R}^4_+\setminus \{x\in\mathbb{R}^4_+:I=Y=F=0\}$.
\end{itemize}
\label{Thm_EquilibriumMaleAbund_withMAT}
\end{theorem}

The equilibrium $EE^\#$ is an equilibrium of (\ref{EqModTemp}) if and only if
\[
Y^*+Y_P<\gamma M^\#(Y_P)
\]
or, equivalently,
\begin{equation}
Y_P<Y_P^*:=\frac{\gamma M^*-Y^*}{1+\frac{\alpha}{\mu_M}}=\frac{1}{\mu_M+\alpha}\left(\gamma (1-r)\nu_I-\frac{r \nu_I (\delta+\mu_F)\mu_M}{(\nu_Y+\mu_Y)(\delta+\mu_F)-\delta \nu_Y}\right)\left(1-\frac{1}{\mathcal{N}_0}\right)K.
\label{YPstar}
\end{equation}
The threshold value $Y_P^*$ determines the minimal level of control below which the control has essentially no effect on an established pest population. More precisely, the effect is limited to reducing the number of males, while all other compartments remain in their natural equilibrium.

% % % % % % % % % % % % % % % % % % % % % % % % % % % % % % % % % % % % % %
\subsubsection{Properties of the equilibria in the male scarcity region ($\gamma M<Y+Y_P$)}\label{SubsubMaleScare}

In the region of male scarcity $\gamma M\leq Y+Y_P$ the dynamics of the population is governed by the system:

\begin{equation}
\left\lbrace
	\begin{array}{lcl}
	\frac{dI}{dt} & = & b \left(1-\frac{I}{K} \right)F-\left( \nu_I+\mu_I \right)I, \\
	& &\\
	\frac{dY}{dt} & = & r\nu_I I -\nu_Y\frac{\gamma MY}{Y+Y_P}+\delta F- \mu_{Y} Y, \\
	& &\\
	\frac{dF}{dt} & = & \nu_Y\frac{\gamma MY}{Y+Y_P} -(\delta +\mu_{F})F, \\
	& &\\
	\frac{dM}{dt} & = & (1-r)\nu_I I - (\mu_M+\alpha\frac{Y_P}{Y+Y_P})M.
	\end{array}
	\label{EqModTempScarce}
		\right.
\end{equation}
The following theorem exhibits different behaviours of the model depending on the value of $Y_P$ which can be interpreted as the effort of the mating disruption control.  

\begin{theorem}\label{theoScarceWithMat}
\begin{itemize}
\item[a)] The system of ODEs (\ref{EqModTempScarce}) defines a positive dynamical system on $\mathbb{R}^4_+$.
\item[b)] TE is an asymptotically stable equilibrium of this system.
\item[c)] There exists a threshold value $Y_P^{**}$ of $Y_P$ such that
\begin{itemize}
\item[i)] if $Y_P>Y_P^{**}$ the only equilibrium of the system on $\mathbb{R}^4_+$ is TE;
\item[ii)] if $0<Y_P<Y_P^{**}$ the system has three biologically relevant equilibria on $\mathbb{R}^4_+$, TE and two positive equilibria.
\end{itemize}
\end{itemize}
\label{Thm_SystemWithMATmaleScarcity}
\end{theorem}

\begin{proof}
a) The local existence of the solutions of (\ref{EqModTempScarce}) follows from the fact that the right hand side is Lipschitz continuous in $\mathbb{R}^4_+$. To obtain global existence it is enough to show that every solution is bounded. This can be proved directly, but it also follows from the upper approximation of the solutions discussed in the proof of Theorem~\ref{Thm_GAS_TE}. Hence, it is omitted here.

b) It is clear that $TE=(0,0,0,0)^{T}$ is an equilibrium.
The Jacobian of the right hand side of system (\ref{EqModTempScarce}) is
\begin{equation}
\mathcal{J}_h(x)=
\left(
\begin{array}{cccc}
-\left( \nu_I+\mu_I +b\frac{F}{K} \right)& 0 & b \left(1-\frac{I}{K} \right) & 0 \\
r\nu_I & -(\nu_Y\gamma M\frac{Y_P}{(Y+Y_P)^2}+\mu_Y) & \delta & -\nu_Y\gamma \frac{Y}{(Y+Y_P)} \\
0 & \nu_Y \gamma M \frac{Y_P}{(Y+Y_P)^2} & -(\mu_F+\delta) & \nu_Y\gamma\frac{Y}{(Y+Y_P)} \\
(1-r)\nu_I& \frac{\alpha Y_P}{(Y+Y_P)^2}M & 0 & -(\mu_M+\alpha\frac{Y_P}{Y+Y_P}) \\
\end{array}
\right),
\label{Eq_JacobianMatrix}
\end{equation}
thus,
\begin{equation}
\mathcal{J}_h(TE)=
\left(
\begin{array}{cccc}
-\left( \nu_I+\mu_I \right)& 0 & b & 0 \\
r\nu_I & -\mu_Y & \delta & 0 \\
0 & 0 & -(\delta+\mu_F) & 0\\
(1-r)\nu_I& 0 & 0 & -(\mu_M+\alpha) \\
\end{array}
\right).
\end{equation}
Its eigenvalues are equal to its diagonal entries and are all negative real values. Hence  $TE$ is asymptotically stable.

c) Setting $\frac{dI}{dt}=0$ in (\ref{EqModTempScarce}) yields
\begin{equation}
F=\frac{\nu_I+\mu_I}{b\left(1-\frac{I}{K}\right)}I.\nonumber
\end{equation}
Then, from $\frac{dM}{dt}=0$, we have
\begin{equation}
M=\frac{(1-r)\nu_I}{\mu_M+\alpha\frac{Y_P}{Y+Y_P}}I.\nonumber
\end{equation}
Further, considering that $\frac{dY}{dt}+\frac{dF}{dt}=0$, we deduce
\begin{equation*}
Y=\frac{r \nu_I I - \mu_F F}{\mu_Y}=\left(\frac{r\nu_I}{\mu_Y}-\frac{\mu_F(\nu_I+\mu_I)}{\mu_Yb\left(1-\frac{I}{K}\right)} \right)I=\frac{\phi(I)}{\mu_Yb\left(1-\frac{I}{K}\right)}I.
\end{equation*}
with $\phi(I)=r\nu_Ib(1-\frac{I}{K})-\mu_F(\nu_I+\mu_I)$.
Now, we use $\frac{dF}{dt}=0$,
\begin{equation}
\frac{dF}{dt} =\frac{\nu_Y \gamma M}{Y+Y_P}Y-(\delta+\mu_F)F
=\frac{\nu_Y \gamma (1-r)\nu_I I}{\mu_M(Y+Y_P)+\alpha Y_P}\frac{\phi(I)}{\mu_Yb\left(1-\frac{I}{K}\right)}I-(\delta+\mu_F)\frac{\nu_I+\mu_I}{b\left(1-\frac{I}{K}\right)}I=0. \nonumber
\end{equation}
Thus by substituting the expressions of $Y$, $F$ and $M$, we have
\begin{equation}
\begin{array}{rl}
 & \frac{\nu_Y \gamma (1-r)\nu_I I}{\mu_M Y+(\mu_M+\alpha)Y_P}\frac{\phi(I)}{\mu_Yb\left(1-\frac{I}{K}\right)}I-(\delta+\mu_F)\frac{\nu_I+\mu_I}{b\left(1-\frac{I}{K}\right)}I=0\\
 &\\
\Leftrightarrow & \left( \nu_Y \gamma (1-r)\nu_I \frac{\phi(I)}{\mu_Y b \left(1-\frac{I}{K}\right)}I \right)=\frac{(\delta+\mu_F)(\nu_I+\mu_I)}{b\left(1-\frac{I}{K}\right)}(\mu_M Y+(\mu_M+\alpha)Y_P) \\
&\\
\Leftrightarrow & \left( \nu_Y \gamma (1-r)\nu_I \frac{\phi(I)}{\mu_Y b \left(1-\frac{I}{K}\right)}I -\frac{(\delta+\mu_F)(\nu_I+\mu_I)(\mu_M)}{b\left(1-\frac{I}{K}\right)}\frac{\phi(I)}{\mu_Y b \left(1-\frac{I}{K}\right)}\right)=\frac{(\delta+\mu_F)(\nu_I+\mu_I)(\mu_M+\alpha)Y_P}{b\left(1-\frac{I}{K}\right)}.
\end{array}\nonumber
\end{equation}
Multiplying both side by $\mu_Y b^2(1-\frac{I}{K})^2$, we obtain an equation for $I$ in the form
\begin{equation}
\psi(I):=I\xi(I)\phi(I)=\eta(Y_P,I),
\label{EqPoly}
\end{equation}
where
\begin{equation}
\xi(I)=\nu_Y \gamma (1-r)\nu_Ib\left(1-\frac{I}{K}\right)-(\delta+\mu_F)(\nu_I+\mu_I)\mu_M,
\label{Eq_xsi}
\end{equation}
\begin{equation}
\phi(I)=r\nu_Ib\left(1-\frac{I}{K}\right)-\mu_F(\nu_I+\mu_I),
\end{equation}
and
\begin{equation}
\eta(Y_P,I)=\mu_Y(\delta+\mu_F)(\nu_I+\mu_I)(\mu_M+\alpha)b\left(1-\frac{I}{K}\right)Y_P.
\label{Eq_eta}
\end{equation}
Therefore, the non-trivial equilibria of (\ref{EqModTempScarce}) are of the form
\begin{eqnarray}
Y_{MD} & = & \left(\frac{r\nu_I}{\mu_Y}-\frac{\mu_F(\nu_I+\mu_I)}{\mu_Yb\left(1-\frac{I_{MD}}{K}\right)} \right)I_{MD},\label{Eq_YMD} \label{Y_MD} \\
F_{MD} & = &\frac{\nu_I+\mu_I}{b\left(1-\frac{I_{MD}}{K}\right)}I_{MD},  \\
M_{MD} & = & \frac{(1-r)\nu_I}{\mu_M+\alpha\frac{Y_P}{Y_{MD}+Y_P}}I_{MD}.
\end{eqnarray}
with $I_{MD}$ a positive root of (\ref{EqPoly}).
Further, we note that to ensure $Y_{MD}>0$, it follows from (\ref{Eq_YMD}) that $I_{MD}$ must satisfy the condition
\begin{equation}
I_{MD}<K\left(1-\frac{\mu_F(\nu_I+\mu_I)}{r\nu_Ib}\right).
\label{Ineq_IMD}
\end{equation}
Thus, to obtain biologically viable equilibria, $I_{MD}$ must belong to the interval $\left[0, K\left(1-\frac{\mu_F(\nu_I+\mu_I)}{r\nu_Ib}\right)\right]$.
\label{Rmk_IMAT}

The roots of  (\ref{EqPoly}) correspond to the values of $I$ where the graph of the cubic polynomial $\psi$ and the straight line $\eta(Y_P,\cdot)$ intersect. It is clear that, the straight line $\eta(Y_P,\cdot)$ intersects the $I$-axis at $I=K$.  
From the factorization of $\psi(I)$ in (\ref{EqPoly}), it is clear that the roots of $\psi$ are
\begin{equation}
I_0=0, \quad I_1=K\left(1-\frac{(\delta+\mu_F)(\nu_I+\mu_I)\mu_M}{\nu_Y\gamma(1-r)\nu_Ib}\right), \textrm{ and } I_2=K\left(1-\frac{\mu_F(\nu_I+\mu_I)}{r\nu_Ib}\right).
\label{Roots_psi}
\end{equation}
Note that the the roots $I_1$ and respectively, $I_2$, are positive and smaller than $K$, i.e.
\begin{equation}
0<I_1,I_2<K,\nonumber
\end{equation}
provided
\begin{equation}
\frac{(\delta+\mu_F)(\nu_I+\mu_I)\mu_M}{\nu_Y\gamma(1-r)\nu_Ib}<1,
\label{Eq_ConRoots1}
\end{equation}
and respectively,
\begin{equation}
\frac{\mu_F(\nu_I+\mu_I)}{r\nu_Ib}<1.
\label{Eq_ConRoots2}
\end{equation}
Inequality (\ref{Eq_ConRoots1}) is satisfied under assumptions (\ref{Assumption2}) and (\ref{Assumption1}), and inequality (\ref{Eq_ConRoots2}) is equivalent to $\mathcal{N}_0>1$, that is, to (\ref{Assumption1}). Therefore, under assumptions (\ref{Assumption1}) and (\ref{Assumption2}), we have that
\begin{equation}
I_1, I_2 \in [0,K].\nonumber
\end{equation}
Therefore, the graph of the cubic polynomial $\psi$ is as given on Figure~\ref{Fig_GraphPolyIntersectOrigMod}.
\begin{figure}[H]
\centering
\includegraphics[width=8cm]{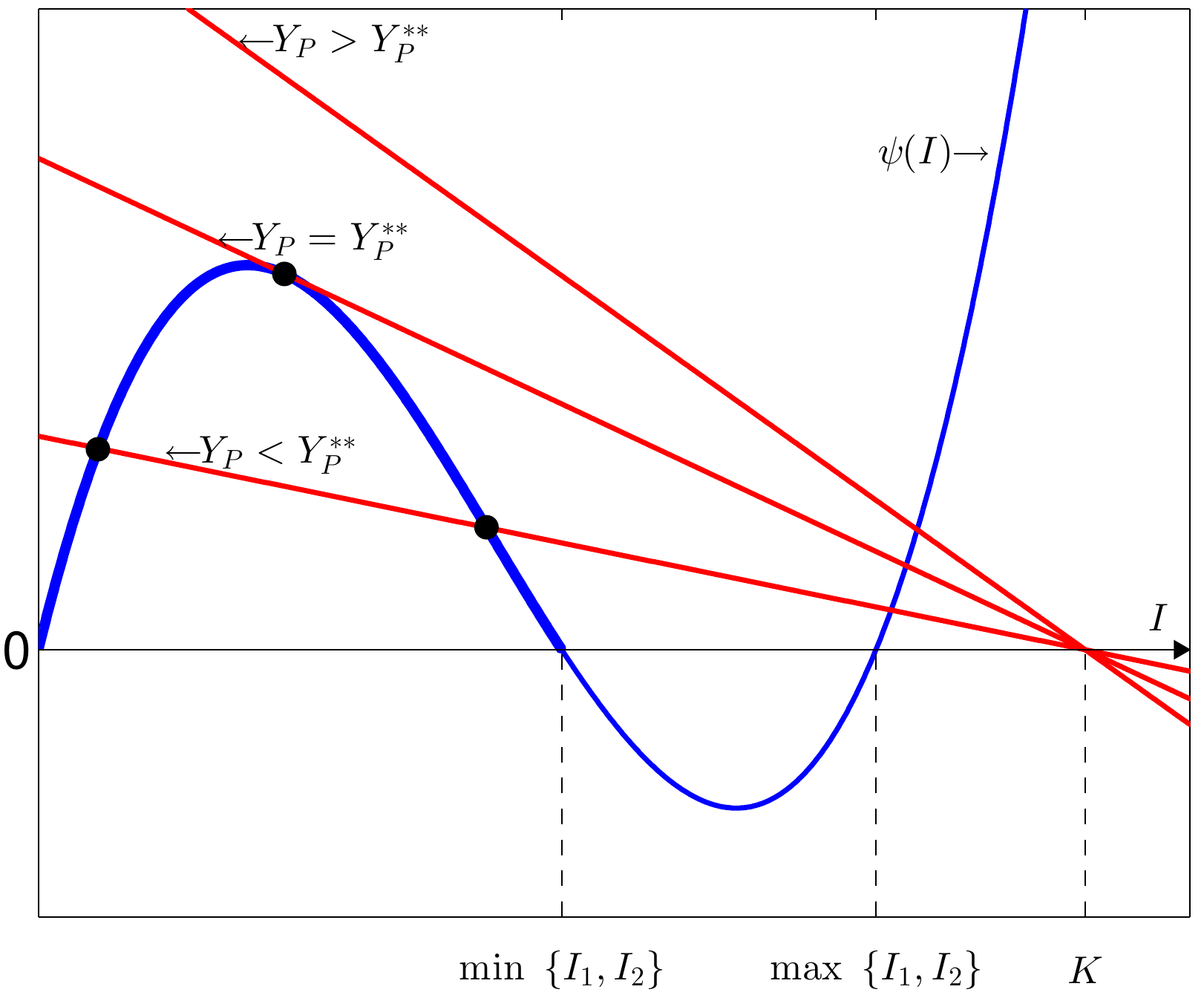}
\caption{Intersections between the graphs of $\eta(Y_P,\cdot)$ (in red) and $\psi$ (in blue) for different values of $Y_P$. The black dots represent the intersection points on the interval $[0,\min\{I_1,I_2\}]$.}
\label{Fig_GraphPolyIntersectOrigMod}
\end{figure}
Considering the inequality (\ref{Ineq_IMD}), only points of intersection of the straight line $\eta(Y_P,\cdot)$ with the section of the graph of $\psi$ for $0\leq I \leq \min\{I_1,I_2\}$, indicated by a thicker line on Figure~\ref{Fig_GraphPolyIntersectOrigMod}, are of relevance to the equilibria of the model.
Let us note that $\psi$ is independent of $Y_P$ while the gradient of the line $\eta$ is a multiple of $Y_P$. We denote by $Y_P^{**}$ the value of $Y_P$ such that the line $\eta(Y_P,\cdot)$ is tangent to the indicated section of the graph of $\psi$, see Figure~\ref{Fig_GraphPolyIntersectOrigMod}. Then it is clear that for $Y_P>Y_P^{**}$ there is no intersection between $\eta(Y_P,\cdot)$ and $\psi$ on $\left[0,\min\{I_1,I_2\}\right]$ while if $0<Y_P<Y_P^{**}$ there are two such points of intersection. This proves items (i) and (ii) in c).

\end{proof}

The next step in this analysis is to merge the results in Theorem~\ref{Thm_EquilibriumMaleAbund_withMAT} for system (\ref{EqModTempAbundance}) and the results in Theorem~\ref{Thm_SystemWithMATmaleScarcity} for the system (\ref{EqModTempScarce}) in order to obtain results for the model (\ref{EqModTemp}) which is actually our interest.

Let $I^{(1)}_{MD}$ and $I^{(2)}_{MD}$,$I^{(1)}_{MD}<I^{(2)}_{MD}$, be the roots of (\ref{EqPoly}) when $0<Y_P<Y_P^{**}$ and denote the respective equilibria by $EE^{(1)}_{MD}$ and $EE^{(2)}_{MD}$.
First we show that $EE^{\#}=EE^{(2)}_{MD}$ when $Y_P=Y_P^*$. Indeed, $Y_P^*$ is selected in such way that $\displaystyle\frac{\gamma M^*}{Y^*+Y^*_P}=1$. Then, the right hand sides of (\ref{EqModTempAbundance}) and (\ref{EqModTempScarce}) are the same at $EE^{\#}$. This implies that $EE^{\#}$ is an equilibrium of (\ref{EqModTempScarce}).

\begin{figure}[H]
\centering
\includegraphics[width=8cm]{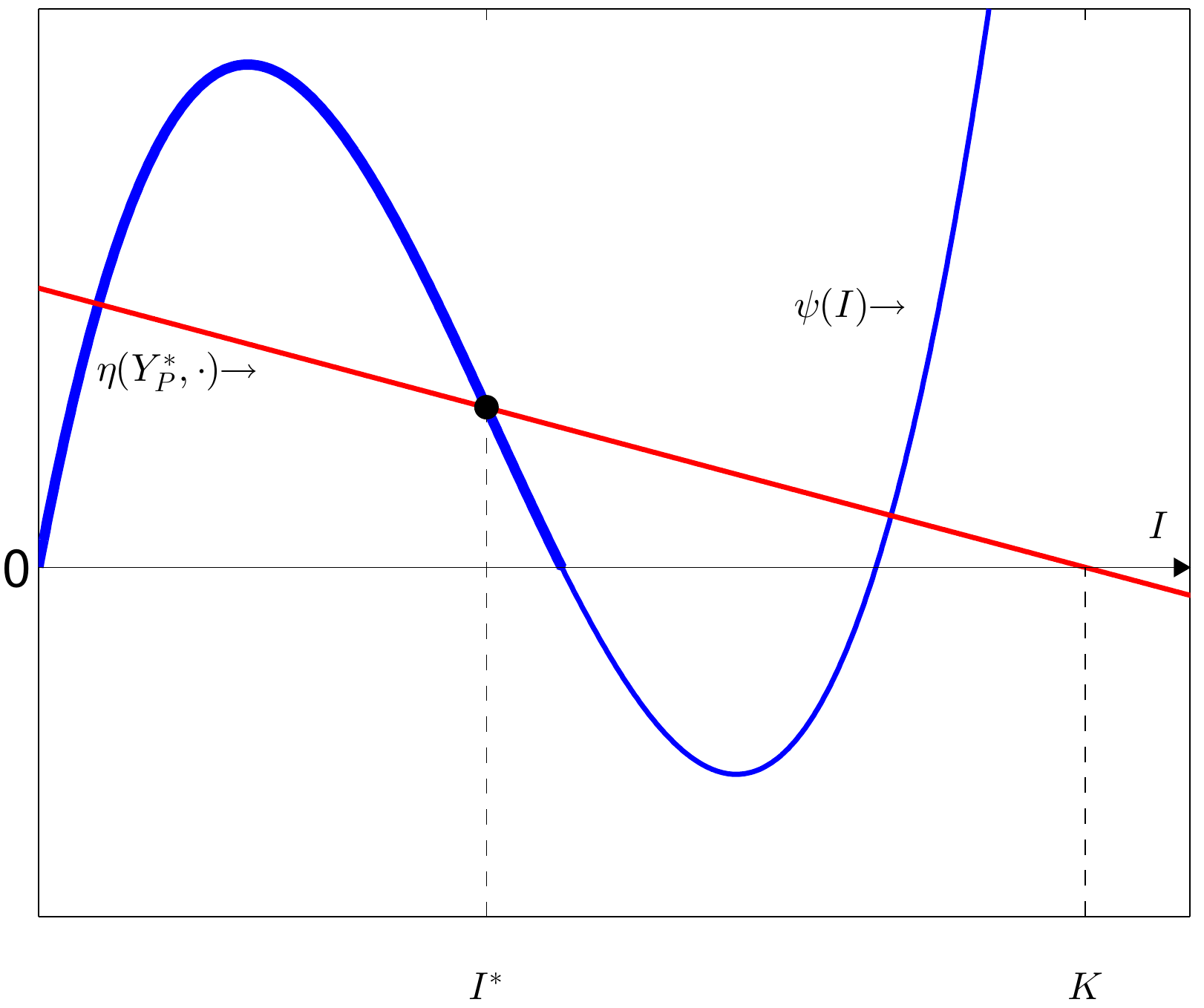}
\caption{Intersections between the graphs of $\eta(Y_P^*,\cdot)$ (in red) and $\psi$ (in blue). The black dot represent the intersection for $I=I^*$.}
\label{Fig_GraphPolyIntersectOrigModbis}
\end{figure}
One can see from Figure~\ref{Fig_GraphPolyIntersectOrigModbis}
 that for $0<Y_P<Y_P^*$, we have 
 \begin{equation}
I^{(1)}_{MD}<I^*<I^{(2)}_{MD},
\label{IneqIstar0}
 \end{equation}
 and for $Y_P^*<Y_P<Y_P^{**}$, we have
 \begin{equation}
 I^{(1)}_{MD}<I^{(2)}_{MD}<I^*.
 \label{IneqIstar}
  \end{equation}
We investigate when the equilibria of (\ref{EqModTemp}) are biologically relevant, that is when they belong to the male scarcity region. We have
\begin{eqnarray}
& & Y_{MD}+Y_P-\gamma M_{MD} \\ 
& & = Y_{MD}+Y_P-\gamma\frac{(1-r)\nu_I}{\mu_M+\alpha\frac{Y_P}{Y_{MD}+Y_P}}I_{MD}\nonumber \\
& & =\frac{\mu_MY_{MD}+\mu_Y Y_P+\alpha Y_P -\gamma (1-r)\nu_I}{\mu_M+\alpha\frac{Y_P}{Y_{MD}+Y_P}}I_{MD}. \nonumber \\
& & =\left(\frac{\mu_{M}}{\mu_M+\alpha\frac{Y_P}{Y_{MD}+Y_P}}\left(\frac{r\nu_I}{\mu_Y}-\frac{\mu_F(\nu_I+\mu_I)}{\mu_Y\left(1-\frac{I_{MD}}{K}\right)}\right)-\gamma(1-r)\nu_I \right)I_{MD}+(\mu_M+\alpha)Y_P.\nonumber
\end{eqnarray}

Let $Y_P^*<Y_P<Y_P^{**}$. Then using also (\ref{IneqIstar}) we have
\begin{eqnarray}
& & Y_{MD}+Y_P-\gamma M_{MD} \nonumber  \\
& &\geq\frac{\mu_{M}}{\mu_M+\alpha\frac{Y_P}{Y_{MD}+Y_P}}\left(\frac{r\nu_I}{\mu_Y}-\frac{\mu_F(\nu_I+\mu_I)}{\mu_Y\left(1-\frac{I^*}{K} \right)}-\gamma(1-r)\nu_I\right)I_{MD}+(\mu_M+\alpha)Y_P\label{Ineq1} \\
& & =\left(\left(\mu_M\frac{r \nu_I (\delta+\mu_F)}{\nu_Y \mu_F + \mu_Y \mu_ F +\delta \mu_Y}\right)-\gamma(1-r)\nu_I\right)I_{MD}+(\mu_M+\alpha)Y_P \quad \textrm{(using (\ref{YPstar}))} \nonumber \\
&& =-(\mu_M+\alpha)\frac{Y_P^*}{I^*}I_{MD}+(\mu_M+\alpha)Y_P\nonumber \\
&& = \frac{\mu_M+\alpha}{I^*}\left(Y_P I^*-Y_P^*I_{MD}\right)>0.\nonumber
\end{eqnarray}
Therefore, in this case $EE^{(1)}_{MD}$ and $EE^{(2)}_{MD}$ are both in the male scarcity region. Hence, they are also equilibria of (\ref{EqModTemp}).

If $Y_P<Y_P^*$ and $I_{MD}>I^*$, considering (\ref{IneqIstar0}), then using the same method as in (\ref{Ineq1}) we obtain
\begin{equation}
Y_P+Y_{MD}-\gamma M_{MD}<0.
\end{equation}
Therefore, $EE^{(2)}_{MD}$ is not in the male scarcity region. Hence, it is not an equilibrium of (\ref{EqModTemp}).

\

Taking into consideration the above results regarding $EE^{(1)}_{MD}$ and $EE^{(2)}_{MD}$ we obtain the following theorem for the model (\ref{EqModTemp}).

\begin{theorem}
Let $Y_P>0$. The following holds for model (\ref{EqModTemp}):
\begin{itemize}
\item[a)] $TE$ is an asymptotically stable equilibrium.
\item[b)] If $0<Y_P<Y_P^*$ there are two positive equilibria $EE^{(1)}_{MD}$ and $EE^{\#}$, where $EE^{\#}$ is asymptotically stable.
\item[c)] If $Y_P^*<Y_P<Y_P^{**}$ there are two positive equilibria $EE^{(1)}_{MD}$ and $EE^{(2)}_{MD}$.
\item[d)] If $Y_P>Y_P^{**}$ there is no positive equilibrium.
\end{itemize}
\label{Thm_SummaryEqModTemp}
\end{theorem}

Obtaining theoretically the stability properties of the equilibria $EE^{(1)}_{MD}$ and $EE^{(2)}_{MD}$ is not easy considering the complexity of the system.
The numerical simulations indicate that $EE^{(1)}_{MD}$ is unstable while $EE^{(2)}_{MD}$ is asymptotically stable, and that the equilibria are the only invariant set of the system on $\mathbb{R}^{+}_{4}$. Further, when $Y_P>Y_P^{**}$, $TE$ is globally asymptotically stable.
The equilibria and their properties are presented on the bifurcation diagram in Figure~\ref{Fig_BifurcationDiagram}. 
The equilibrium values of $Y+F$ are given as function of the bifurcation parameter $Y_P$. The solid line represents stable equilibria, while the dotted line represents unstable equilibria.

\begin{figure}[H]
\centering
\includegraphics[width=10cm]{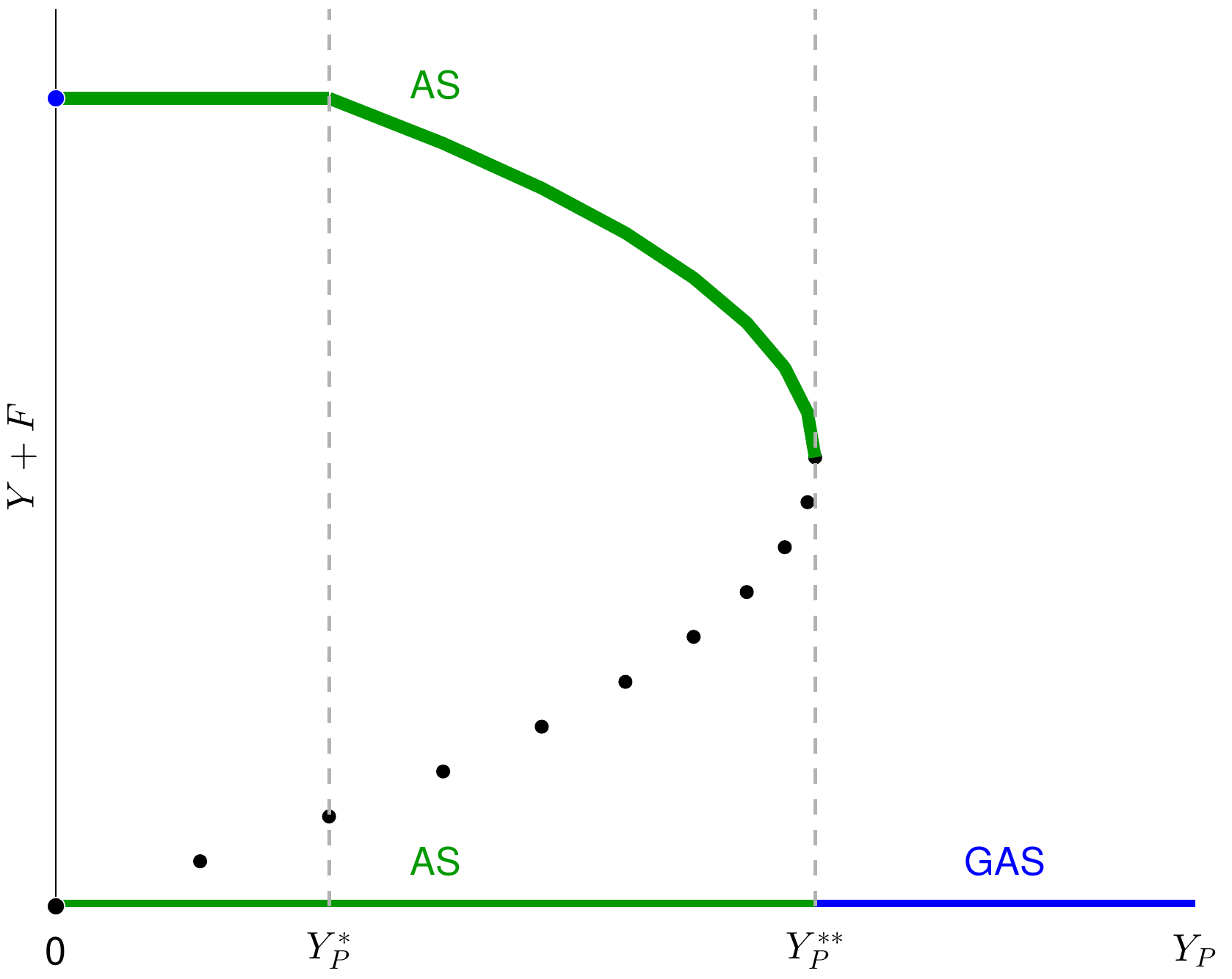}
\caption{Bifurcation diagram of the values of $Y+F$ at equilibrium with respect to the values of $Y_P$ for system (\ref{EqModTemp}).}
\label{Fig_BifurcationDiagram}
\end{figure}

% % % % % % % % % % % % % % % % % % % % % % % % % % % % % % % % % % % % % %
\subsubsection{Global asymptotic stability of the trivial equilibrium for sufficiently large $Y_P$}\label{SubsubMaleScarceGASofTE}

Theorem~\ref{Thm_SummaryEqModTemp} shows that for $Y_P<Y_P^{**}$ the insect population persists at substantial endemic level. Hence, the numerically observed global asymptotic stability of $TE$ for $Y_P>Y_P^{**}$ is of significant practical importance.
This section deals with the mathematical proof of this result. More precisely, we will establish global asymptotic stability of $TE$ under slightly stronger condition $Y_P>\tilde{Y}_P^{**}$ where $\tilde{Y}_P^{**}>Y_P^{**}$. We also show that $\tilde{Y}_P^{**}$ is a close approximation for $Y_P^{**}$.

The asymptotic analysis for system (\ref{EqModTemp}) cannot be conducted in the same way as for the other systems considered so far, since it is not a monotone system. More precisely due to the term $\displaystyle-\nu_Y\frac{\gamma M}{Y+Y_P}Y$ in the equation for the $Y$ compartment, the right hand side of (\ref{EqModTemp}) is not quasi-monotone. In order to obtain the practically important result mentioned above, we consider an auxiliary system which is monotone and provides upper bounds for the solutions of (\ref{EqModTemp}). The system is obtained by removing the mentioned term in the second equation and replacing $\displaystyle\min\{\frac{\gamma M}{Y+Y_P},1\}$ by $\displaystyle\frac{\gamma M}{Y+Y_P}$ in the third equation. In vector form it is given as
\begin{equation}
\frac{dx}{dt}=\tilde{h}(x),
\label{Sys_WithMAT_eta}
\end{equation}
where $x=(I,Y,F,M)^{T}$ and 

\begin{equation}
\tilde{h}(x)=
\left(
\begin{array}{c}
b \left(1-\frac{I}{K} \right)F-\left( \nu_I+\mu_I \right)I\\
r\nu_I I -\mu_Y Y+\delta F \\
\nu_Y\frac{\gamma M}{Y+Y_P}Y -\delta F-\mu_{F}F \\
(1-r)\nu_I I - (\mu_M+\alpha\frac{Y_P}{Y+Y_P})M
\end{array}
\right).
\end{equation}

\begin{theorem}
\begin{itemize}
\item[a)] The system of ODEs (\ref{Sys_WithMAT_eta}) defines a positive dynamical system on $\mathbb{R}^4_+$.
\item[b)] TE is asymptotically stable equilibrium.
\item[c)] There exists a threshold value $\tilde{Y}_P^{**}$ such that
\begin{itemize}
\item[i)] if $Y_P>\tilde{Y}_P^{**}$, TE is globally asymptotically stable on $\mathbb{R}^4_+$;
\item[ii)] if $0<Y_P<\tilde{Y}_P^{**}$, the system has three equilibria, $TE$ and two positive equilibria $\tilde{E}^{(1)}$ and $\tilde{E}^{(2)}$ such that
$\tilde{E}^{(1)}<\tilde{E}^{(2)}$. The basin of attraction of $TE$ contains the set $\{x\in\mathbb{R}^4_+:0\leq x<\tilde{E}^{(1)}\}$. 
The basin of attraction of $\tilde{E}^{(2)}$ contains the set $\{x\in\mathbb{R}^4_+: x\geq \tilde{E}^{(2)}, I\leq K \}$.
\end{itemize}
\end{itemize}
\label{theoModifiedSystem}
\end{theorem}
\begin{proof}
a) and b) are proved similarly to a) and b) in Theorem \ref{theoScarceWithMat}.

c) Setting the first, second and fourth component of $\tilde{h}$ to zero, we have $$Y=\frac{r\nu_I I +\delta F}{\mu_Y},\quad F=\frac{\nu_I+\mu_I}{b\left(1-\frac{I}{K}\right)}I, \quad  M=\frac{(1-r)\nu_I}{\mu_M Y +(\mu_M+\alpha)Y_P}I.$$ 
Setting the third component of $\tilde{h}$ to zero and substituting the expressions for $Y$, $F$ and $M$ above, we obtain an equation for $I$ in the form
\begin{equation}
\tilde{\psi}(I):= I  \xi(I) \tilde{\phi}(I)=\eta(Y_P,I).
\end{equation}
where $\xi(I)$ and $\eta(Y_P,I)$ are the linear expression given in (\ref{Eq_xsi}) and (\ref{Eq_eta}) and
\begin{equation}
\tilde{\phi}(I)=r\nu_Ib\left(1-\frac{I}{K}\right)+\delta(\nu_I+\mu_I).
\label{EqPolyTilde}
\end{equation}
Therefore the non-trivial equilibria of (\ref{Sys_WithMAT_eta}) are of the form
\begin{eqnarray}
\tilde{Y} & = & \frac{1}{\mu_Y}\left(r\nu_I+\frac{\delta(\nu_I+\mu_I)}{b\left(1-\frac{\tilde{I}}{K}\right)} \right)\tilde{I}, \nonumber \\
\tilde{F} & = &\frac{\nu_I+\mu_I}{b\left(1-\frac{\tilde{I}}{K}\right)}\tilde{I},  \label{Equi_eta}\\
\tilde{M} & = & \frac{(1-r)\nu_I}{\mu_M+\alpha\frac{Y_P}{\tilde{Y}+Y_P}}\tilde{I},\nonumber
\end{eqnarray}
with $\tilde{I}$ a positive root of (\ref{EqPolyTilde}).
The roots of (\ref{EqPolyTilde}) correspond to the values of $I$ where the graph of the cubic polynomial $\tilde{\psi}$ and the straight line $\eta(Y_P,\cdot)$ intersect. It is clear that the straight line $\eta(Y_P,\cdot)$ intersects the $I$ axis at $I=K$. Note that the state where $I$ greater is than $K$ is not sustainable, and therefore not biologically relevant. From the factorization of $\tilde{\psi}$ in (\ref{EqPolyTilde}), it is clear that $I_0$ and $I_1$ given in (\ref{Roots_psi}) and
\begin{equation}
\tilde{I}_2=\left(1+\frac{\delta (\nu_I+\mu_I)}{r\nu_Ib}\right)K
\end{equation}
are roots of $\tilde{\psi}$. Clearly, we have
\begin{equation}
0<I_1<K<\tilde{I}_2. \nonumber
\label{IneqImodif}
\end{equation}
Therefore, the graph of the cubic polynomial $\tilde{\psi}$ is as given in Figure~\ref{Fig_GraphPolyIntersect}.
\begin{figure}[H]
\centering
\includegraphics[width=8cm]{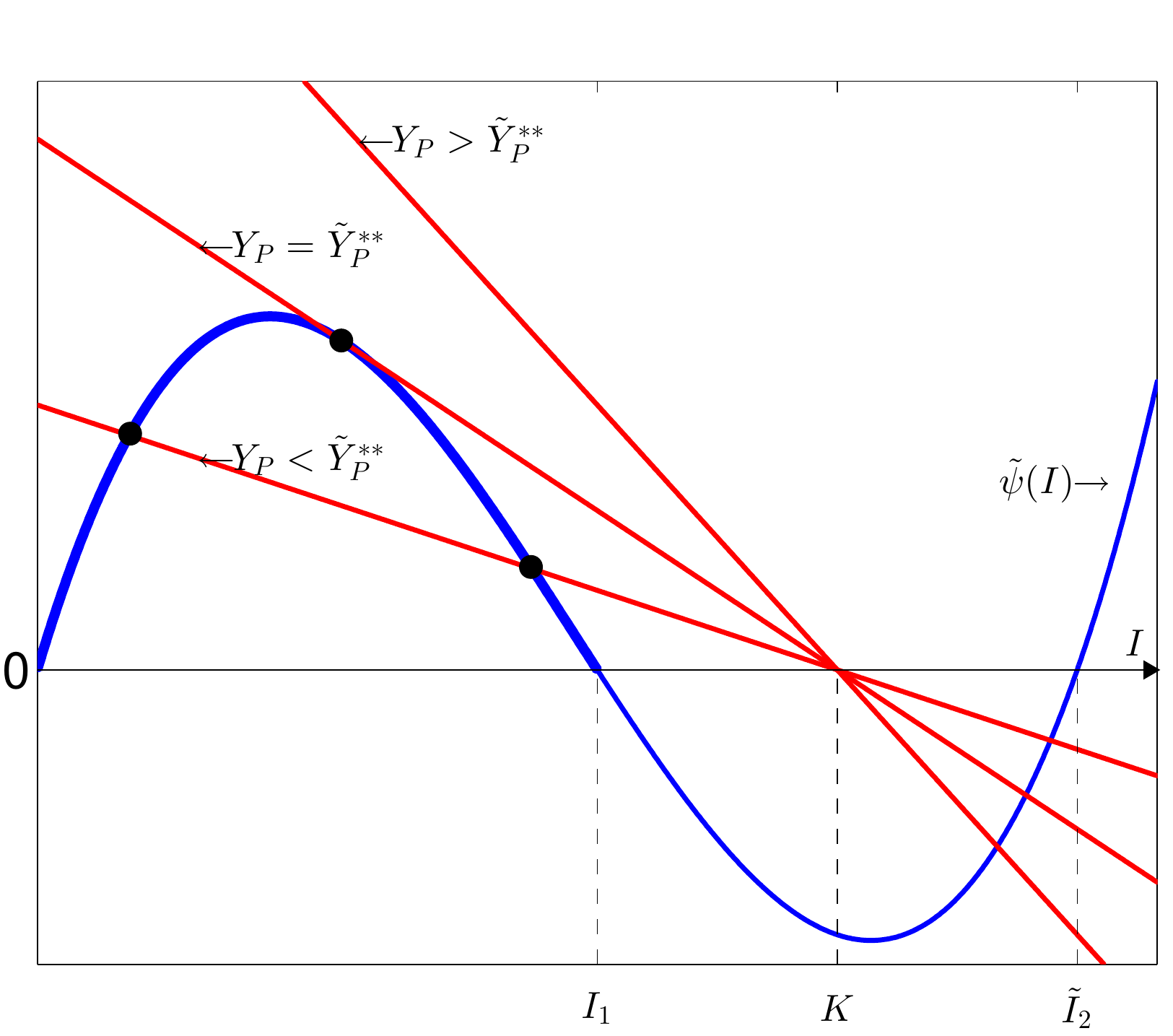}
\caption{Intersections between the graphs of $\eta(Y_P,\cdot)$ (in red) and $\psi$ (in blue) for different values of $Y_P$. The black dots represent the intersection points  on the interval $[0,I_1]$.}
\label{Fig_GraphPolyIntersect}
\end{figure}

Considering the inequalities (\ref{Ineq_IMD}) and (\ref{IneqImodif}), only the points of intersection of the straight line $\eta(Y_P,\cdot)$ with the section of the graph of $\tilde{\psi}$ for $0<I<I_1$ are of relevance to the equilibria of the model.
Note that $\tilde{\psi}$ is independent of $Y_P$ while the gradient of $\eta(Y_P,\cdot)$ is a multiple of $Y_P$. We denote by $\tilde{Y}_P^{**}$ the value of $Y_P$ such that the line $\eta$ is tangent to the indicated section of the graph of $\tilde{\psi}$, see Figure~\ref{Fig_GraphPolyIntersect}. Then, it is clear that for $Y_P>\tilde{Y}_P^{**}$ there is no intersection between $\eta(Y_P,\cdot)$ and $\tilde{\psi}$ on $[0,I_1]$ while if $0<Y_P<\tilde{Y}_P^{**}$ there are two such points of intersection.

i) Let $Y_P>\tilde{Y}_P^{**}$. Consider the point
\begin{equation}
\tilde{y}_q=
\left(
\begin{array}{c}
K\\
\frac{1}{\mu_Y}\left(r\nu_I+\frac{\delta\nu_Y \gamma(1-r)\nu_I}{(\delta+\mu_F)\mu_M}\right)q\\
\frac{\nu_Y \gamma (1-r) \nu_I}{(\delta+\mu_F)\mu_M}q\\
\frac{(1-r)\nu_I }{\mu_M}q\\
\end{array}
\right),
\label{ytilde_q}
\end{equation}
where $q\in\mathbb{R}, q\geq K$. 
It is easy to see that $\tilde{h}(\tilde{y}_q)\leq 0$. Then by Theorem~\ref{Thm_GAS_MonSys}, $TE$ is GAS on $[\textbf{0},\tilde{y}_q]$. Therefore, $TE$ is GAS on $\Omega_K=\cup_{q \geq K}[\textbf{0},\tilde{y}_q]$ as well. Similarly to Theorem~\ref{Prop_PosiDynSysdelta}, $\Omega_K$ is an absorbing set. Hence $TE$ is GAS on $\mathbb{R}^4_+$.

ii) Let $0<Y_P<\tilde{Y}_P^{**}$. Denote the two equilibria by $\tilde{E}^{(j)}=(\tilde{I}^{(j)},\tilde{Y}^{(j)},\tilde{F}^{(j)},\tilde{M}^{(j)})^T$, $j=1,2$, where $\tilde{I}^{(1)}<\tilde{I}^{(2)}$. Since $\tilde{Y}$, $\tilde{F}$, $\tilde{M}$ as given in (\ref{Equi_eta}) are increasing functions of $\tilde{I}$, we have
$0<\tilde{E}^{(1)}<\tilde{E}^{(2)}$ (Figure~\ref{Fig_BassinOfAttraction}).
\begin{figure}[H]
\centering
\includegraphics[width=8cm]{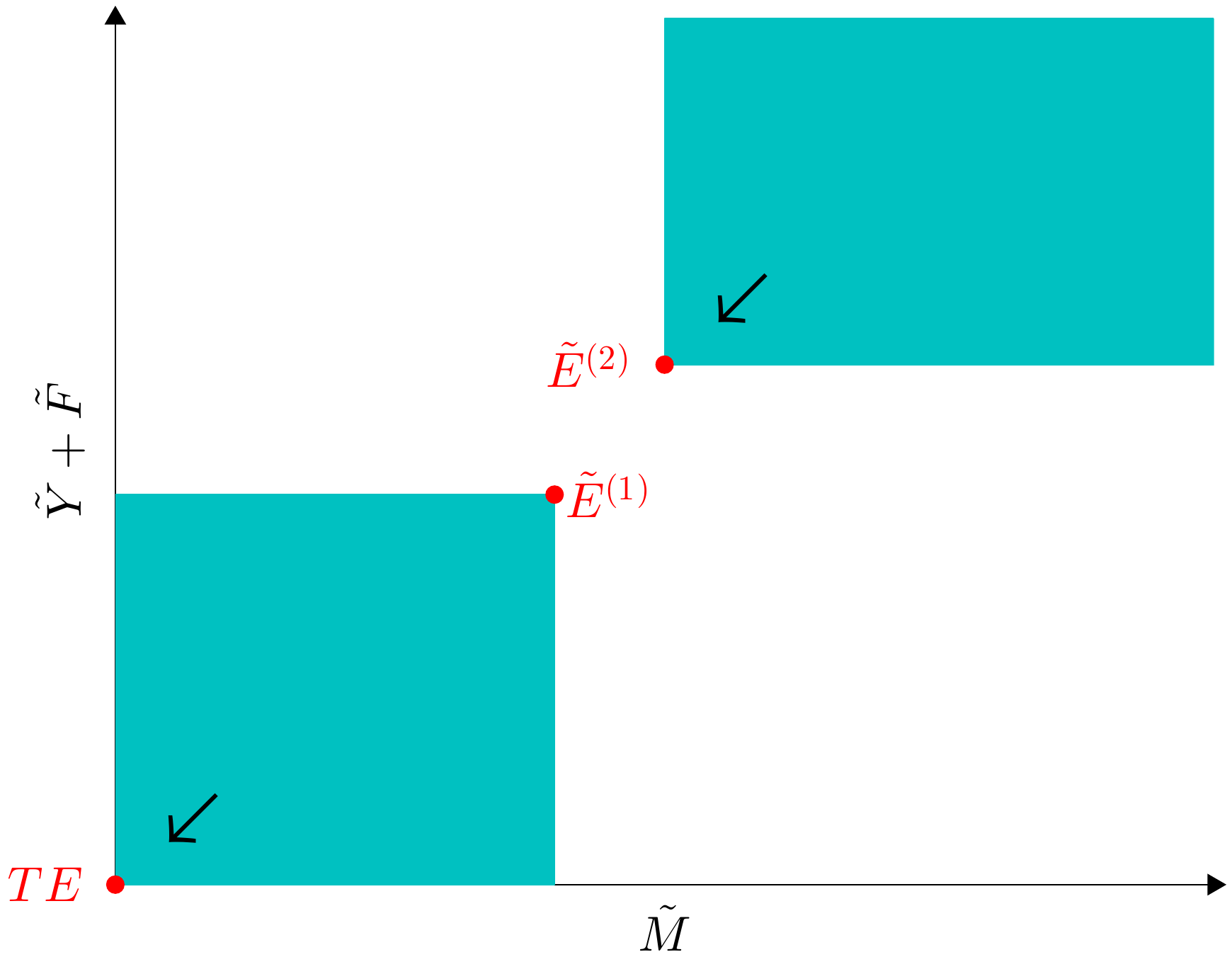}
\caption{Positive invariant sets, when $0<Y_P<\tilde{Y}_P^{**}$.}
\label{Fig_BassinOfAttraction}
\end{figure}
Considering $\tilde{y}_q$ in (\ref{ytilde_q}) we have 
\[
\tilde{h}(\tilde{E}^{(2)})=0\leq \tilde{h}(\tilde{y}_q).
\]
Then by Theorem~\ref{Thm_GAS_MonSys} $\tilde{E}^{(2)}$ is GAS on $[\tilde{E}^{(2)},\tilde{y}_q]$. Therefore, $\tilde{E}^{(2)}$ is GAS on $\cup_{q \geq K}[\tilde{E}^{(2)},\tilde{y}_q]=\{x\in\mathbb{R}^4_+: x\geq \tilde{E}^{(2)},I\leq K\}$. 
We obtain the statements about the basin of attraction of $TE$ by using Theorem~\ref{Thm_BasinAttraction}. Indeed, $TE$ being asymptotically stable, attracts some solutions initiated in $\displaystyle[TE,\tilde{E}^{(1)}]$. Then it follows from Theorem~\ref{Thm_BasinAttraction} that all solutions initiated in $\{x\in\mathbb{R}^4_+: x< \tilde{E}^{(1)}\}$ converge to $TE$.
\end{proof}

The implication of Theorem~\ref{theoModifiedSystem} for the system (\ref{EqModTemp}) are stated in the following theorem which extends the results of Theorem~\ref{Thm_SummaryEqModTemp}.

\begin{theorem}
Let $Y_P>0$. Then the following hold for the model (\ref{EqModTemp}).
\begin{itemize}
\item[a)] If $0<Y_P\leq \tilde{Y}_P^{**}$, then the basin of attraction of $TE$ contains $\{x\in\mathbb{R}^4_+: x<\tilde{E}^{(1)}\}$.
\item[b)] If $Y_P>\tilde{Y}_P^{**}$, then $TE$ is GAS on $\mathbb{R}^4_+$.
\end{itemize}
\label{Thm_GAS_TE}
\end{theorem}

\begin{proof}
Using Theorem~\ref{Thm_Monotonicity} with $x$ being the solution of (\ref{Sys_WithMAT_eta}) and $y$ being the solution of (\ref{EqModTemp}) we obtain that any solution of (\ref{Sys_WithMAT_eta}) is an upper bound of the solution of (\ref{EqModTemp}) with the same initial point. This implies that the basin of attraction of $TE$ as an equilibrium of (\ref{EqModTemp}) contains the sets indicated in Theorem~\ref{theoModifiedSystem} c) (i) and ii)), thus proving a) and b) respectively.
\end{proof}

Theorem~\ref{Thm_GAS_TE} characterises the benefit from the control effort represented by $Y_P$ as follows:

\begin{itemize}
\item Increasing the effort $Y_P$ in the range $0<Y_P<Y_P^{**}$ does not lead to elimination of an established population. In fact, as shown on Figure~\ref{Fig_BifurcationDiagram}, the reduction is not proportional to the effort. However, $\tilde{E}^{(1)}$ increases at least linearly with $\tilde{I}^{(1)}$. Hence, increasing $Y_P$ enlarges the basin of attraction of $TE$, providing better opportunity for controlling an invading population.
\item Effort $Y_P$ stronger than $\tilde{Y}_P^{**}$ eliminates any established or invading population.
\end{itemize}

% % % % % % % % % % % % % % % % % % % % % % % % % % % % % % % % % % % % % %
%  Numerical Simulation and Discussion
% % % % % % % % % % % % % % % % % % % % % % % % % % % % % % % % % % % % % %
\section{Numerical Simulation and Discussion}
We use numerical simulations to illustrate the results of Theorems~\ref{Thm_SummaryEqModTemp} and~\ref{Thm_GAS_TE} and discuss the biological meaning of the results. The numerical simulations are done using the \textbf{ode23tb} solver of Matlab~\cite{MATLAB2012} which solves system of stiff ODEs using a trapezoidal rule and second order backward differentiation scheme (TR-BDF2)~\cite{bank1985transient,hosea1996analysis}. The values of the parameters used for the numerical simulations are those of Table~\ref{Tab_ParaValues}.

\

Using (\ref{NoMAT_BON}), we compute the basic offspring number, $\mathcal{N}_0=122$. $\mathcal{N}_0>1$, therefore, the population establishes to the positive endemic equilibrium $EE^*$ which we have shown to be GAS on $\mathbb{R}^4_+\setminus \{TE\}$ (Theorem~\ref{Prop_PosiDynSysdelta} c)). Figure~\ref{Fig_ExpNoControl} represents the trajectories of a set of solutions of system (\ref{EqModTemp0}), or equivalently system (\ref{EqModTemp}) with $Y_P=0$ and $\alpha=0$, in the $M\times(Y+F)$ plane. The dots represent the points at which the solutions are initiated. The solutions initiated on $\mathbb{R}^4_+\setminus \{TE\}$ all converge to the point $EE^*=(992, 319, 1407, 1498)^T$, represented by the green square. Here, $TE$ is also an equilibrium, but it is unstable (Theorem~\ref{Prop_PosiDynSysdelta} c)), and therefore it is not represented in Figure~\ref{Fig_Exp1}.
\begin{figure}[H]
\centering
\includegraphics[width=7cm]{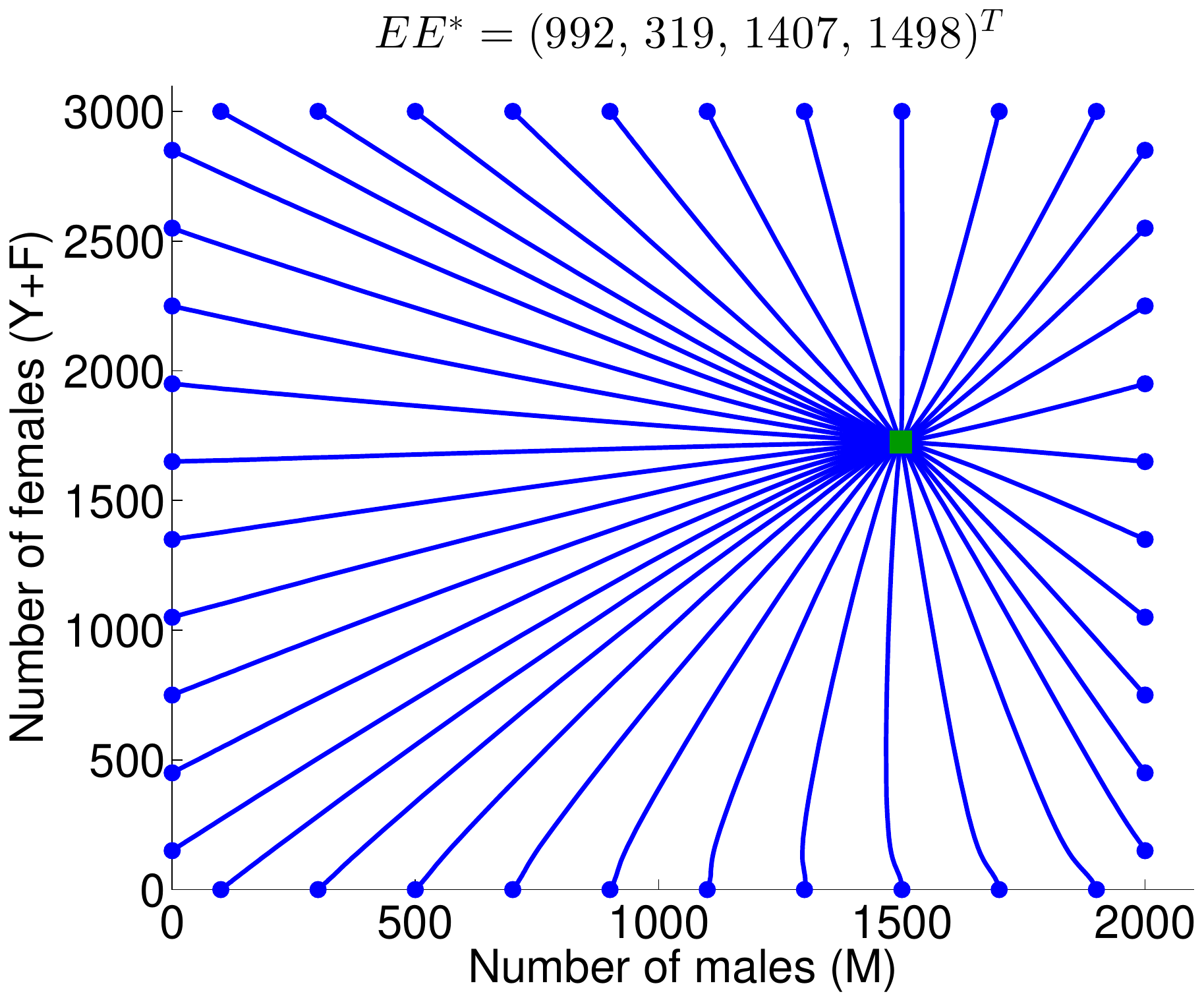}
\caption{Trajectories of a set solutions of system (\ref{EqModTemp0}) in the $M\times (Y+F)$-plane initiated at the dots. The green square represents the stable equilibrium $EE^*$.}
\label{Fig_ExpNoControl}
\end{figure}
In the following, we confirm numerically the theoretical results of Theorems~\ref{Thm_SummaryEqModTemp} and~\ref{Thm_GAS_TE} with respect to the values of the mating disruption thresholds mentioned in those theorems, $Y_P^*$ and $Y_P^{**}$. We also investigate the impact of the trapping effort, $\alpha$, on the population.
Using (\ref{YPstar}) and solving numerically the system
\begin{equation}
\left\lbrace
\begin{array}{ccc}
\psi(I)&=&\eta(Y_P,I)\\
&&\\
\frac{d\psi}{d I}(I)&=&\frac{d\eta}{d I}(Y_P,I)
\end{array}
\right.\nonumber
\end{equation}
with $\psi$ and $\eta(Y_P,\cdot)$ as defined in (\ref{EqPoly}), we obtain respectively the thresholds values $Y_P^*$ and $Y_P^{**}$ as functions of $\alpha$ as represented in Figure~\ref{Fig_ExpYpstarstarVsAlpha}.
\begin{figure}[H]
\begin{center}
\subfigure[$Y_P^*$ as a function of $\alpha$]{
\resizebox*{7cm}{!}{\includegraphics[width=7cm]{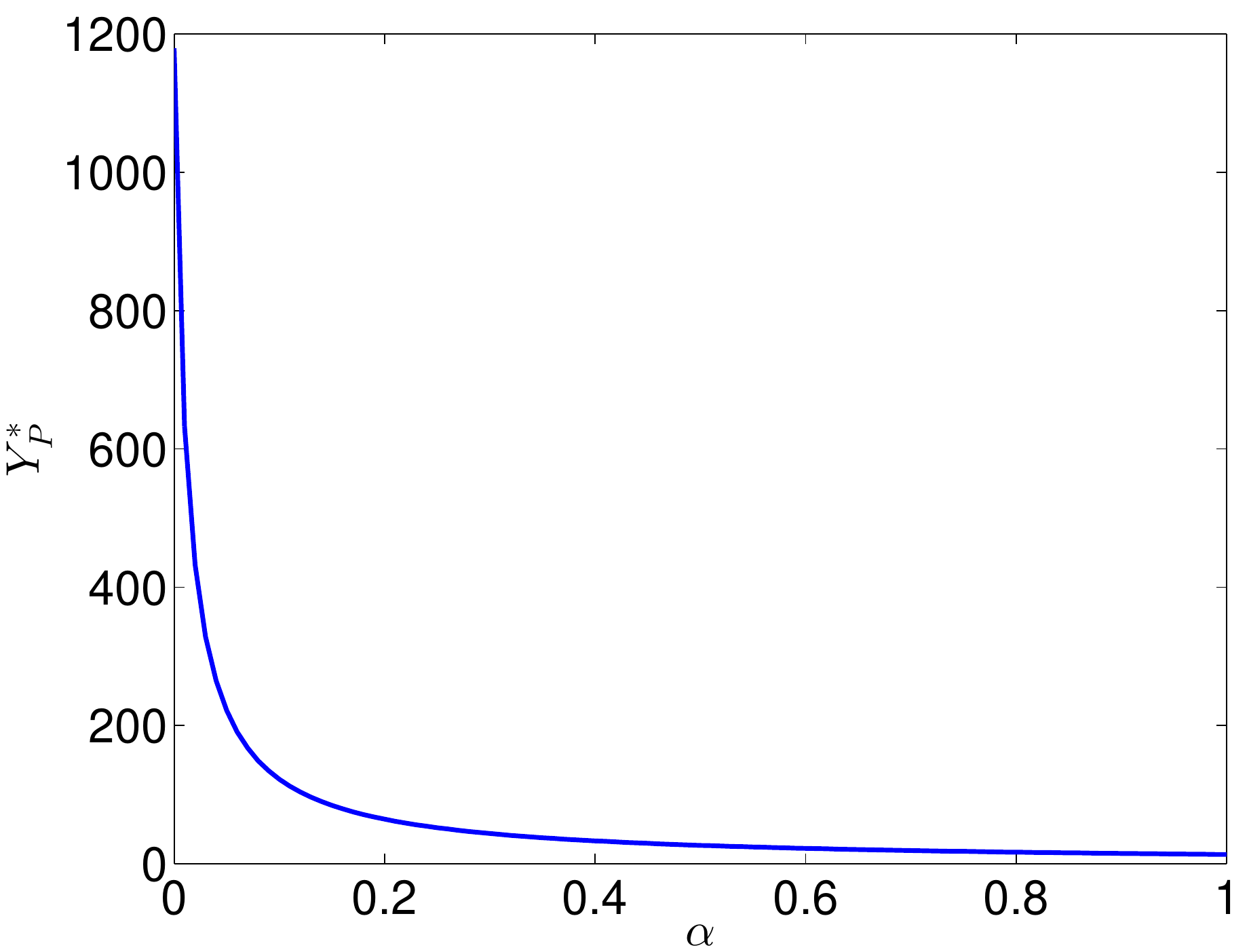}}}
\subfigure[$Y_P^{**}$ as a function of $\alpha$]{
\resizebox*{7cm}{!}{\includegraphics[width=7cm]{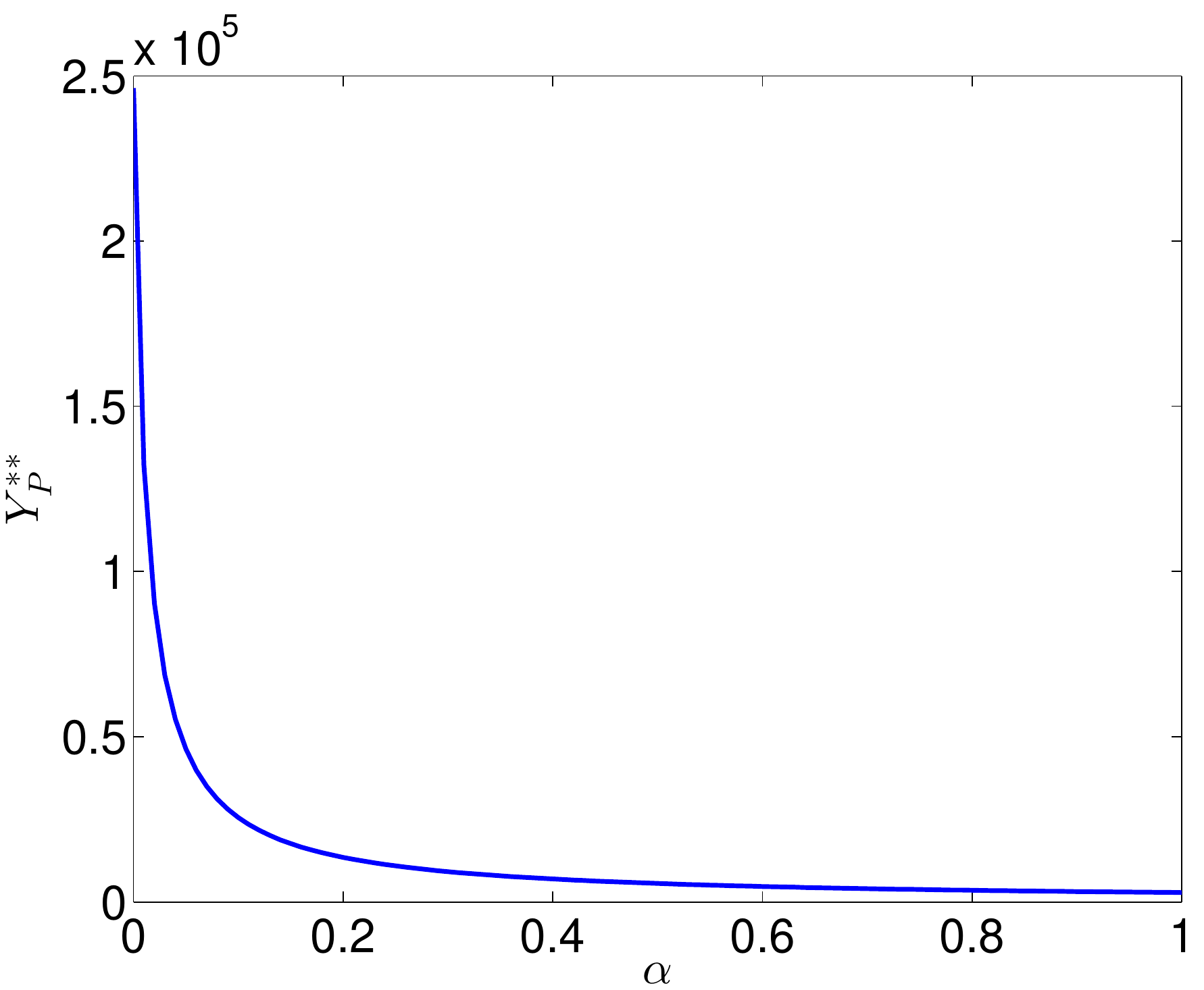}}}\hspace{5pt}
\caption{Mating disruption thresholds as function of the trapping parameter $\alpha$.} 
\label{Fig_ExpYpstarstarVsAlpha}
\end{center}
\end{figure} 
One can observe that a small trapping effort reduces the mating disruption thresholds in a non-linear manner. Adding trapping to mating disruption allows to reduce the amount of the lure for equivalent control efficiency.  
In particular, for $\alpha=0$, we have $Y_P^*=5673$ and $Y_P^{**}=987735$, while for $\alpha=0.1$, we have $Y_P^*=588$ and $Y_P^{**}=102462$. Thus, increasing the trapping effort by $10\%$ reduces both $Y_P^{*}$ and $Y_P^{**}$ by $90\%$. 

\

Figure~\ref{Fig_Exp1} illustrates the trajectories of the solutions of system (\ref{EqModTemp}) when the mating disruption level is below the threshold $Y_P^*$ for $\alpha=0$ and $\alpha=0.1$. The dots represent the initial points and the green squares represent the asymptotically stable equilibria. In this case, the system has one positive equilibrium, $EE^{\#}$ (Theorem~\ref{Thm_SummaryEqModTemp} b)), with value given in the figure. One can observe that when there is no trapping, $EE^{\#}=EE^*$ (Figure~\ref{Fig_Exp1} (a)). This means that the positive endemic equilibrium of the population is not affected by the control. When trapping occurs, we observe in Figure~\ref{Fig_Exp1} (b) that the positive equilibrium is shifted to the left. In fact, the control allows to reduce the number of males but not sufficiently to disrupt the fertilisation of the females. Therefore, the control is not efficient on an established population as the number of females at equilibrium is not reduced.
\begin{figure}[H]
\begin{center}
\subfigure[$\alpha=0$]{
\resizebox*{7cm}{!}{\includegraphics{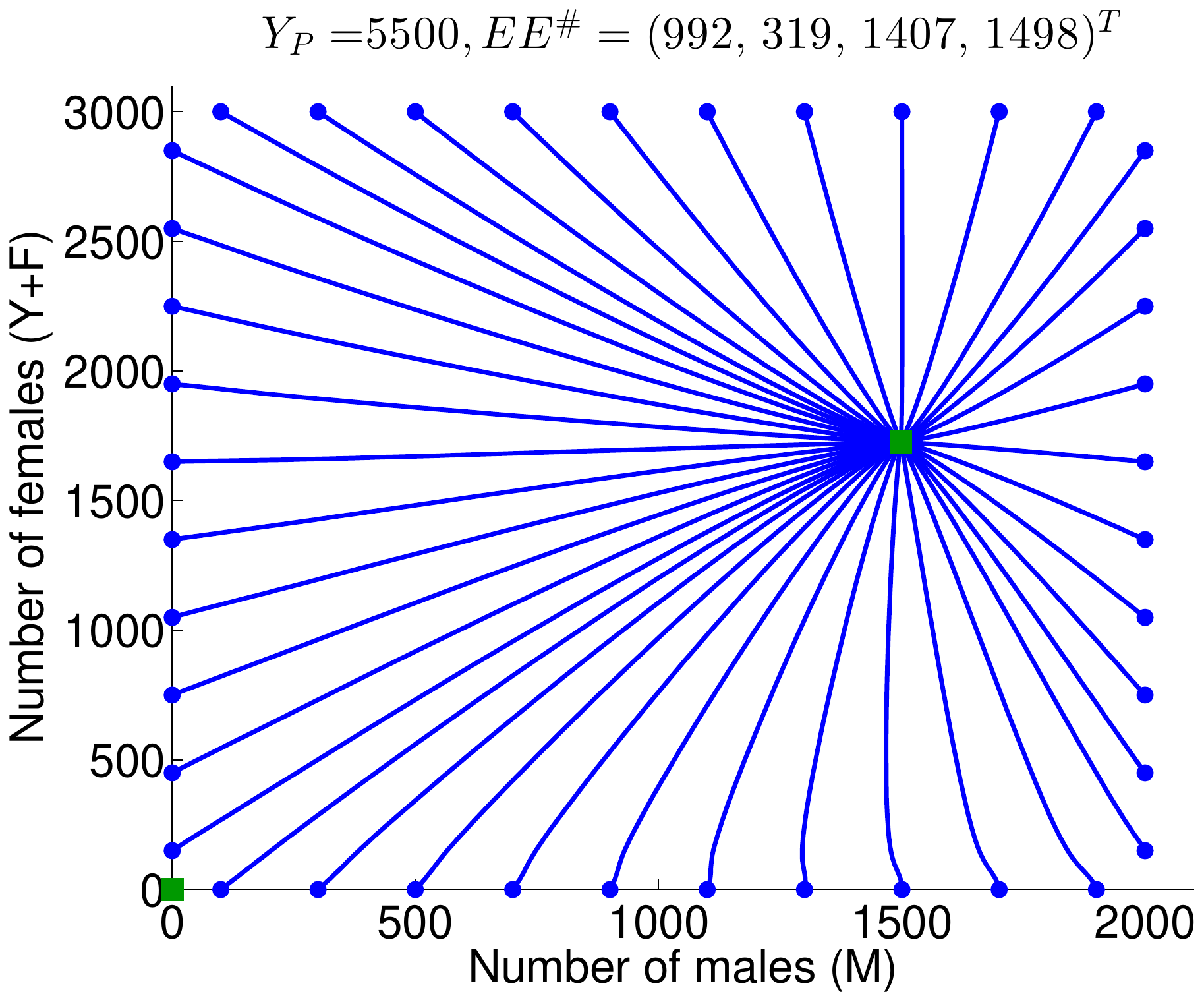}}}
\hspace{0.2cm}
\subfigure[$\alpha=0.1$]{
\resizebox*{7cm}{!}{\includegraphics{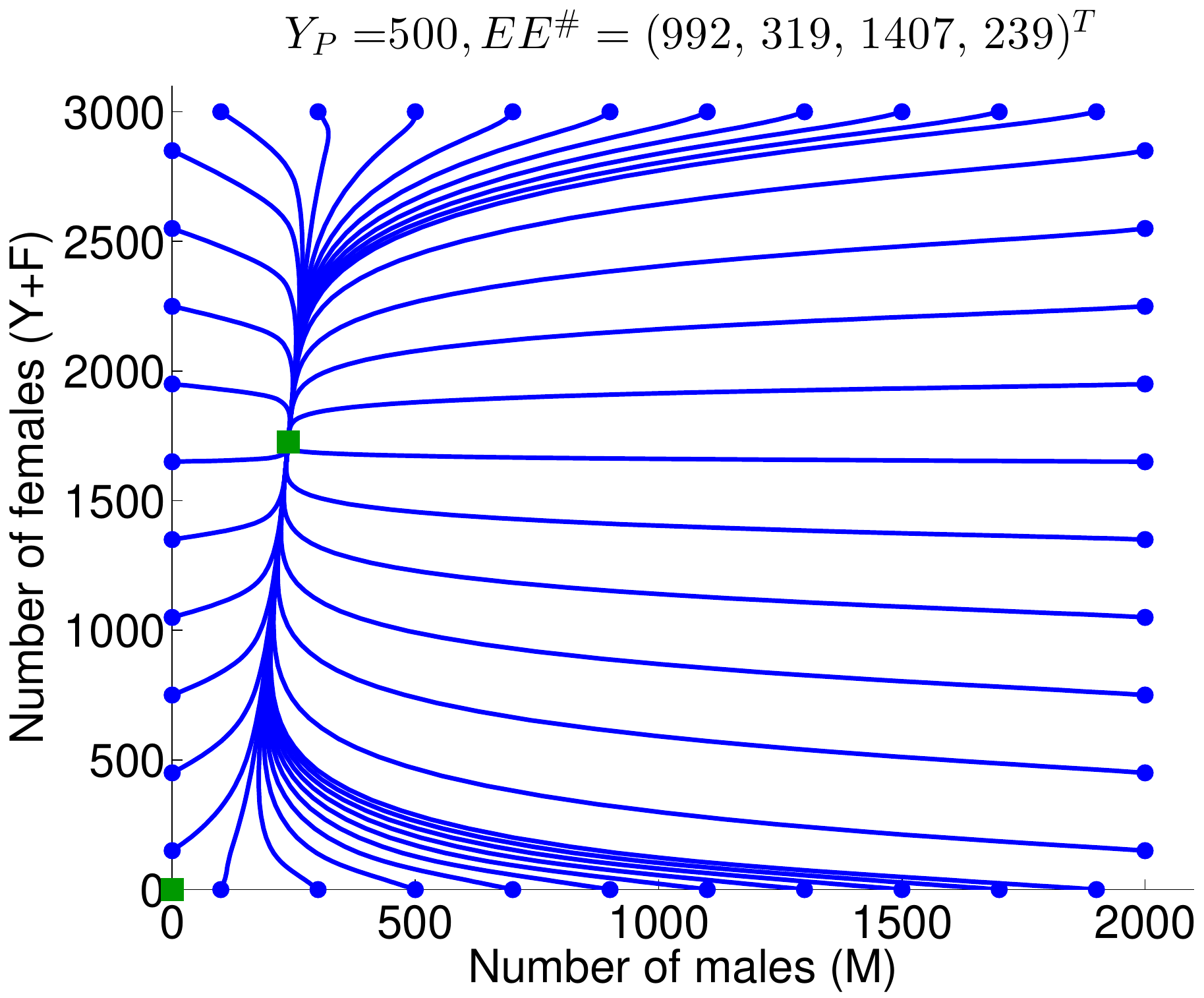}}}\hspace{5pt}
\caption{Trajectories of a set solutions of system (\ref{EqModTemp}) in the $M\times (Y+F)$-plane initiated at the dots. The green squares represent the asymptotically stable equilibria $TE$ and $EE^{\#}$.} 
\label{Fig_Exp1}
\end{center}
\end{figure}
However, when $Y_P>0$, $TE$ is asymptotically stable, which means that a population can be controlled if it is small enough. Figure~\ref{Fig_Exp1bis} (a), represents the basin of attraction of $TE$ (the red dots) for a small population in the same setting as for the experiments in Figure~\ref{Fig_Exp1} (a). We observe that there is a set of solutions, for which the initial population is small enough, that converge to $TE$, hence, a small population can be eradicated for $Y_P>0$. Further, as shown in Figure~\ref{Fig_Exp1bis} (b), adding trapping ($\alpha=0.1$) with the same mating disruption effort as for Figure~\ref{Fig_Exp1bis} (a), enlarges the basin of attraction of $TE$. In other words, for identical mating disruption effort, larger populations can be drawn to extinction when trapping occurs.   
\begin{figure}[H]
\begin{center}
\subfigure[$\alpha=0$.]{
\resizebox*{7cm}{!}{\includegraphics{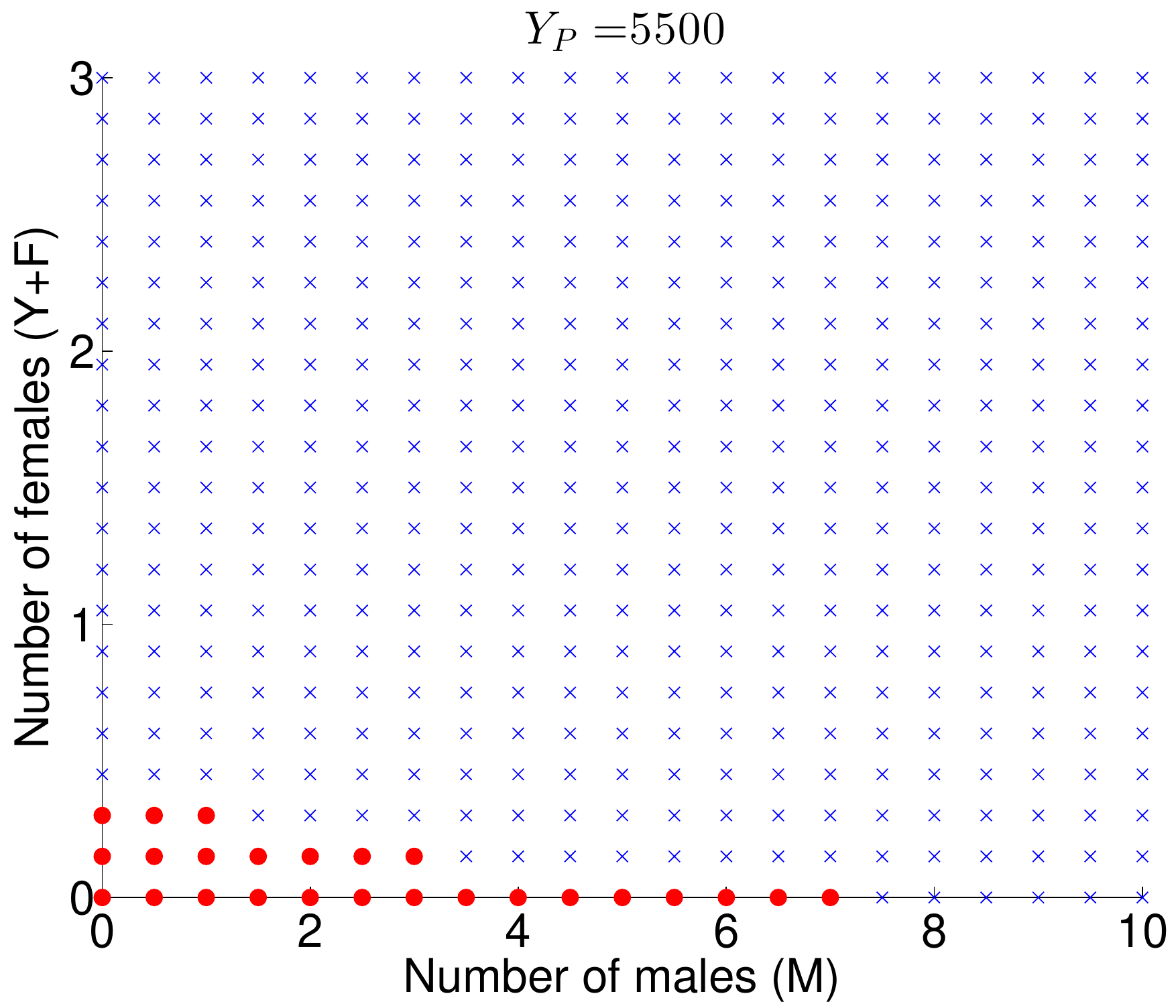}}}\hspace{5pt}
\subfigure[$\alpha=0.1$.]{
\resizebox*{7cm}{!}{\includegraphics{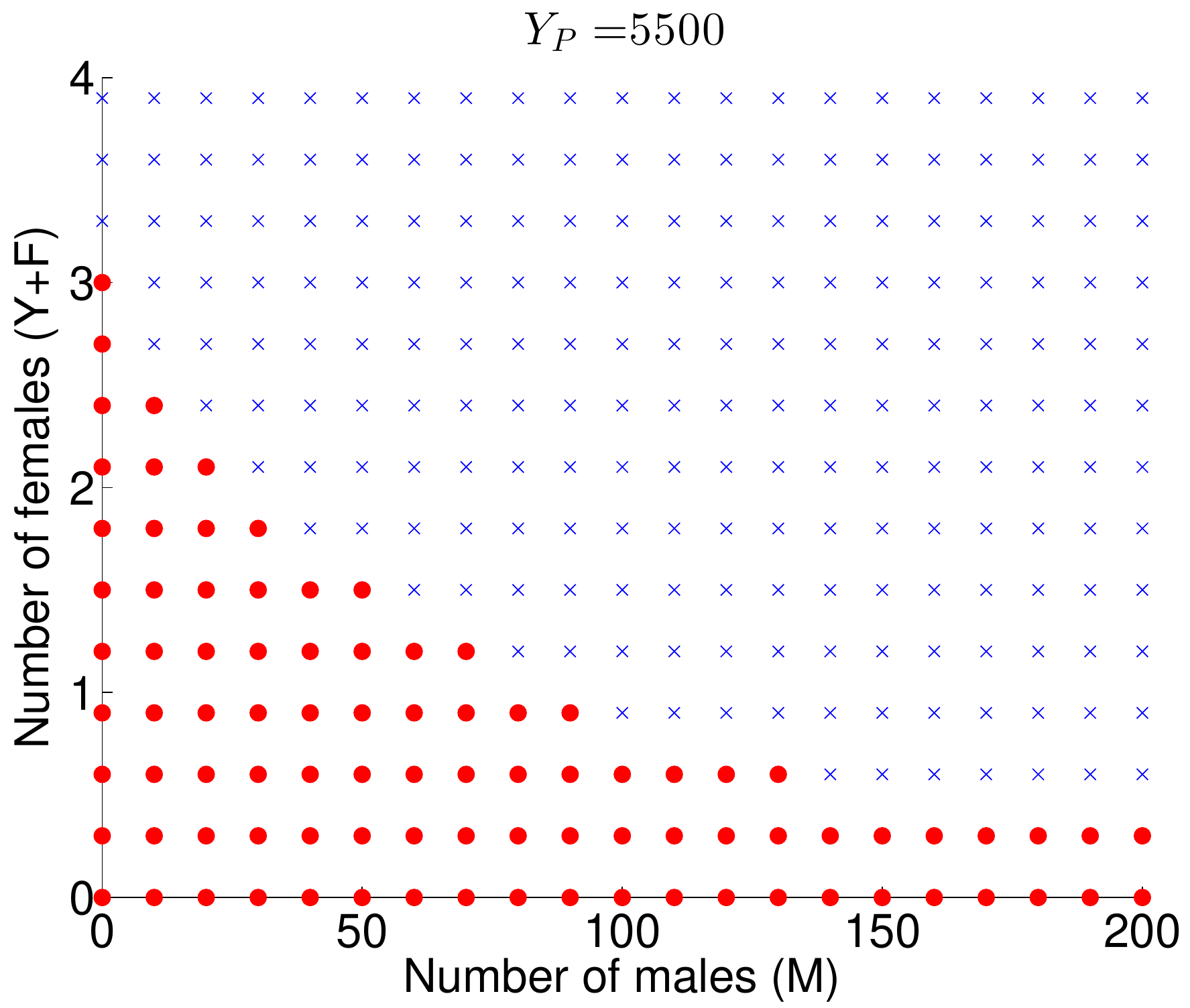}}}
\caption{Effect of $\alpha$ on the basin of attraction of $TE$. The solutions initiated at the points represented by the red dots converge to $TE$ while the blue crosses represent initial points for which the solution converges to $EE^{\#}$.} 
\label{Fig_Exp1bis}
\end{center}
\end{figure}

\

In order to observe a reduction in the number of females, the mating disruption effort has to be increased above the threshold $Y_P^*$. This is the case in Figure~\ref{Fig_Exp2} ((a) and (b)), where $Y_P=0.9999\times Y_P^{**}<Y_P^*$, for $\alpha=0$ and $\alpha=0.1$, respectively. We can see that there is a positive asymptotically stable equilibrium, $EE^{(2)}_{MD}$ represented with a green square (Theorems~\ref{Thm_SummaryEqModTemp} c)).  
The blue lines represent the trajectories of the solutions of (\ref{EqModTemp}) initiated at the blue points which converge to $EE^{(2)}_{MD}$, while the red lines represent the trajectories of the solutions initiated at the red points which converge to $TE$.
With $Y_P>Y_P^*$, the number of females at equilibrium is reduced, however, the impact of the control is not proportional to the effort. Indeed comparing the experiments in Figures~\ref{Fig_Exp1} and Figures~\ref{Fig_Exp2}, we observe that in order to reduce the number of females at equilibrium by $49\%$, the amount of the lure has to be increased by $17857\%$ when $\alpha=0$ and by $20390\%$ when $\alpha=0.1$.   
\begin{figure}[H]
\begin{center}
\subfigure[$\alpha=0$.]{
\resizebox*{7cm}{!}{\includegraphics{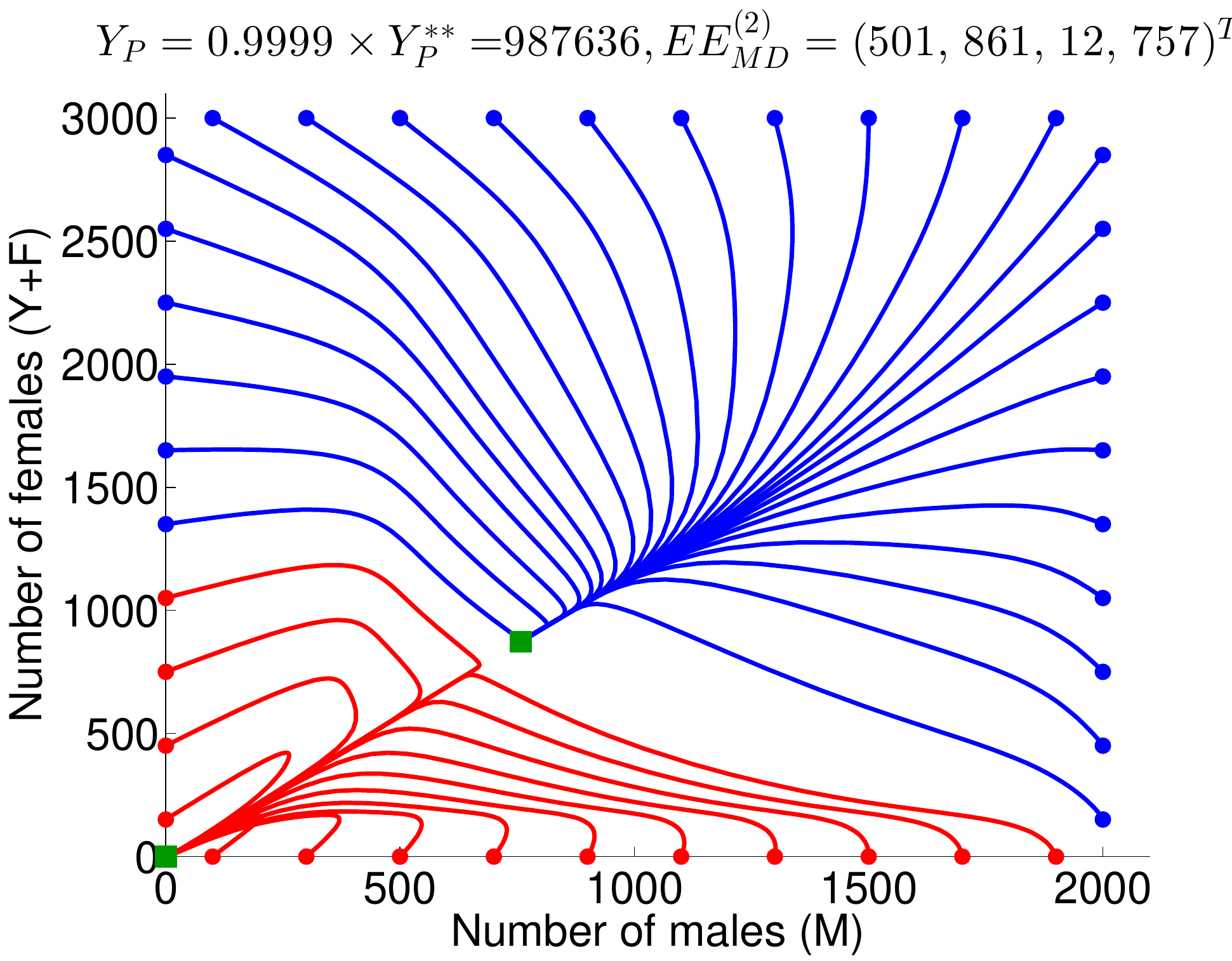}}}\hspace{5pt}
\subfigure[$\alpha=0.1$]{
\resizebox*{7cm}{!}{\includegraphics{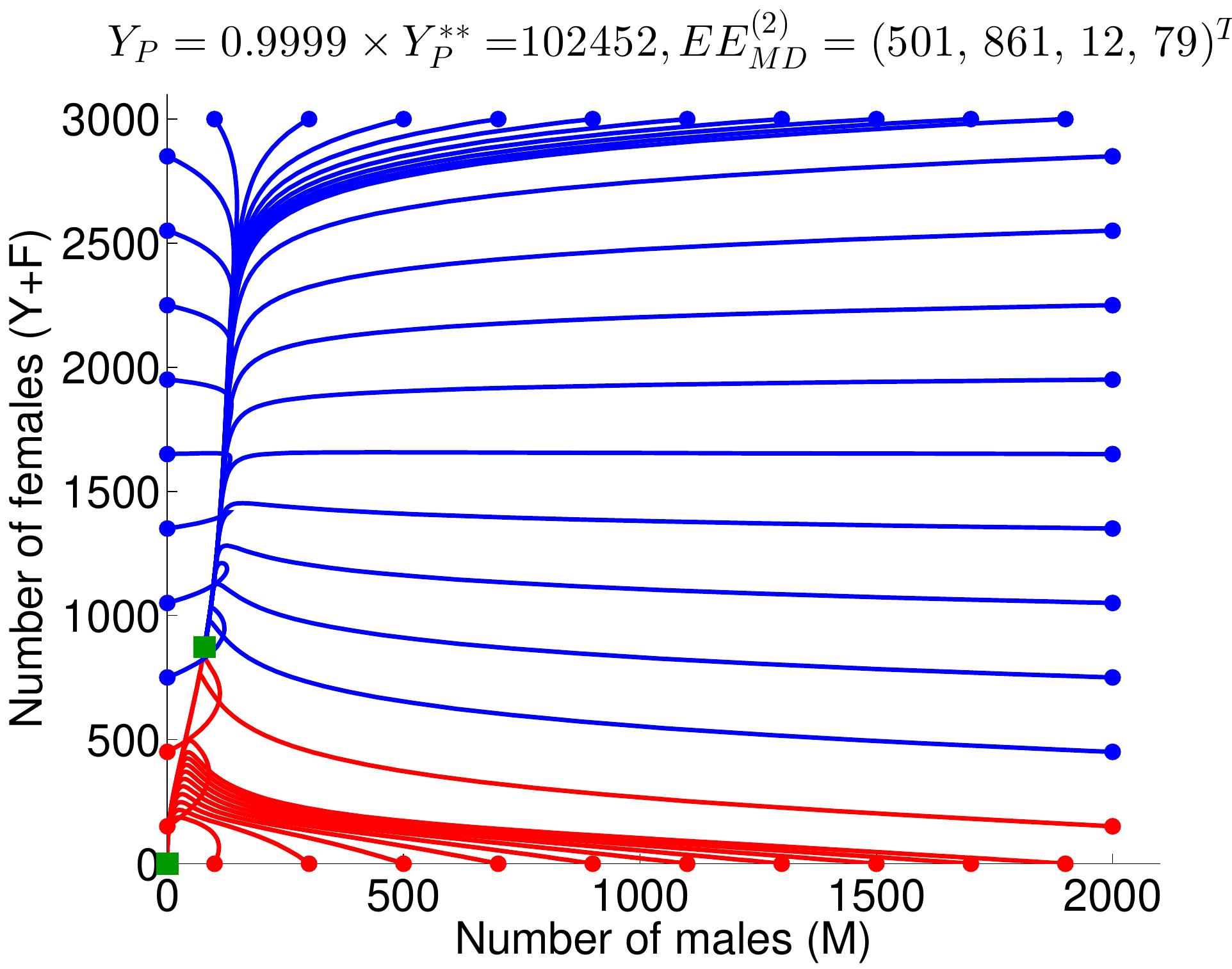}}}
\caption{Trajectories of a set solutions of system (\ref{EqModTemp}) in the $M\times (Y+F)$-plane initiated at the dots. The green squares represent the asymptotically stable equilibria.} 
\label{Fig_Exp2}
\end{center}
\end{figure}
Further, comparing the red dots in Figure~\ref{Fig_Exp1bis} and the red trajectories in Figure~\ref{Fig_Exp2}, we can see that the basin of attraction of the trivial equilibrium becomes larger as the value of $Y_P$ increases.

\

Finally, $Y_P>Y_P^{**}$ allows a full control of the population leading it to extinction no matter how big it is. In Figure~\ref{Fig_Exp3}, $Y_P$ is above $Y_P^{**}$ by $0.01\%$, and we can see that all the trajectories converge to $TE$. This shows the GAS nature of $TE$ when $Y_P>Y_P^{**}$.

\

\begin{figure}[H]
\begin{center}
\subfigure[$\alpha=0$.]{
\resizebox*{7cm}{!}{\includegraphics{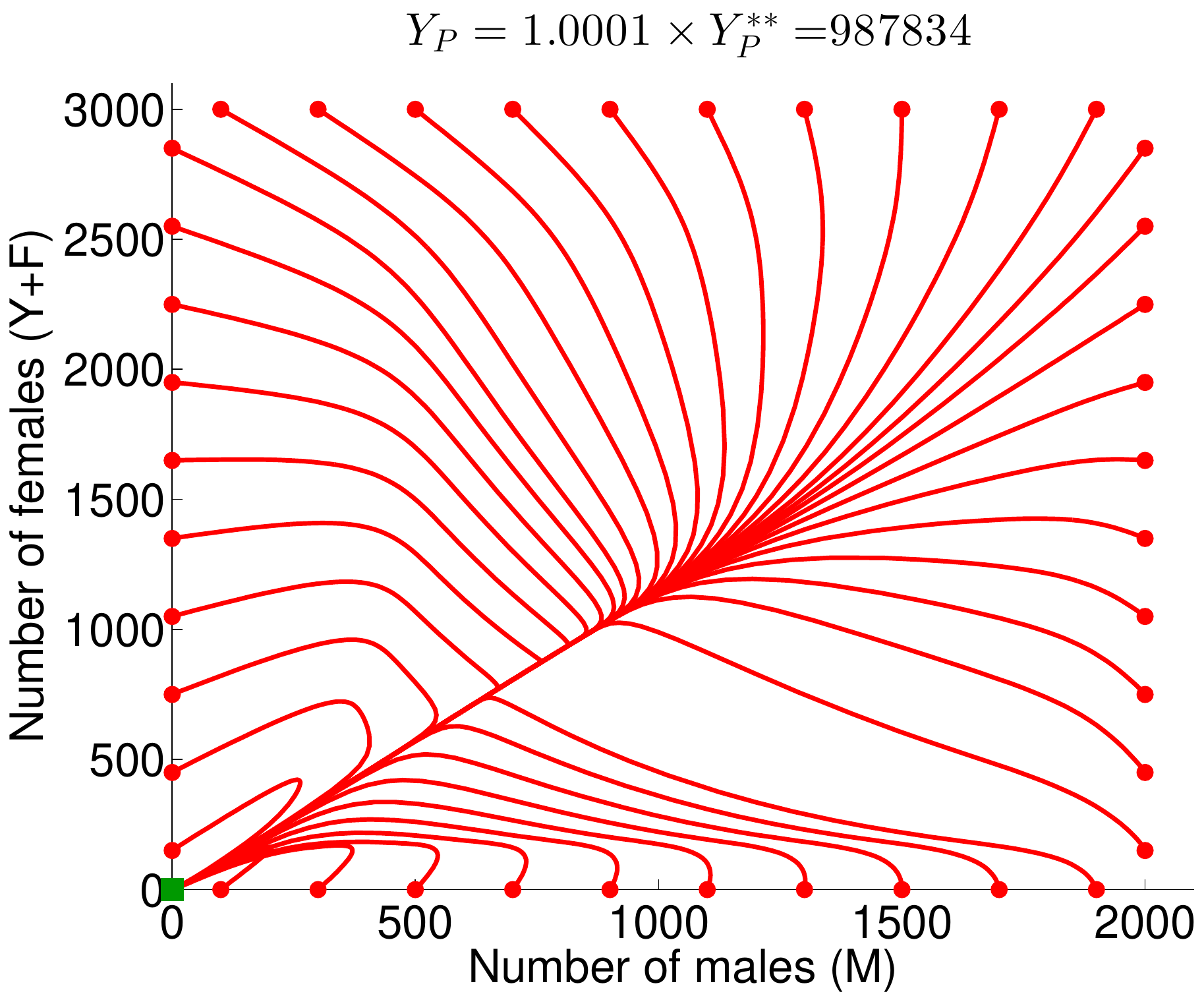}}}\hspace{5pt}
\subfigure[$\alpha=0.1$]{
\resizebox*{7cm}{!}{\includegraphics{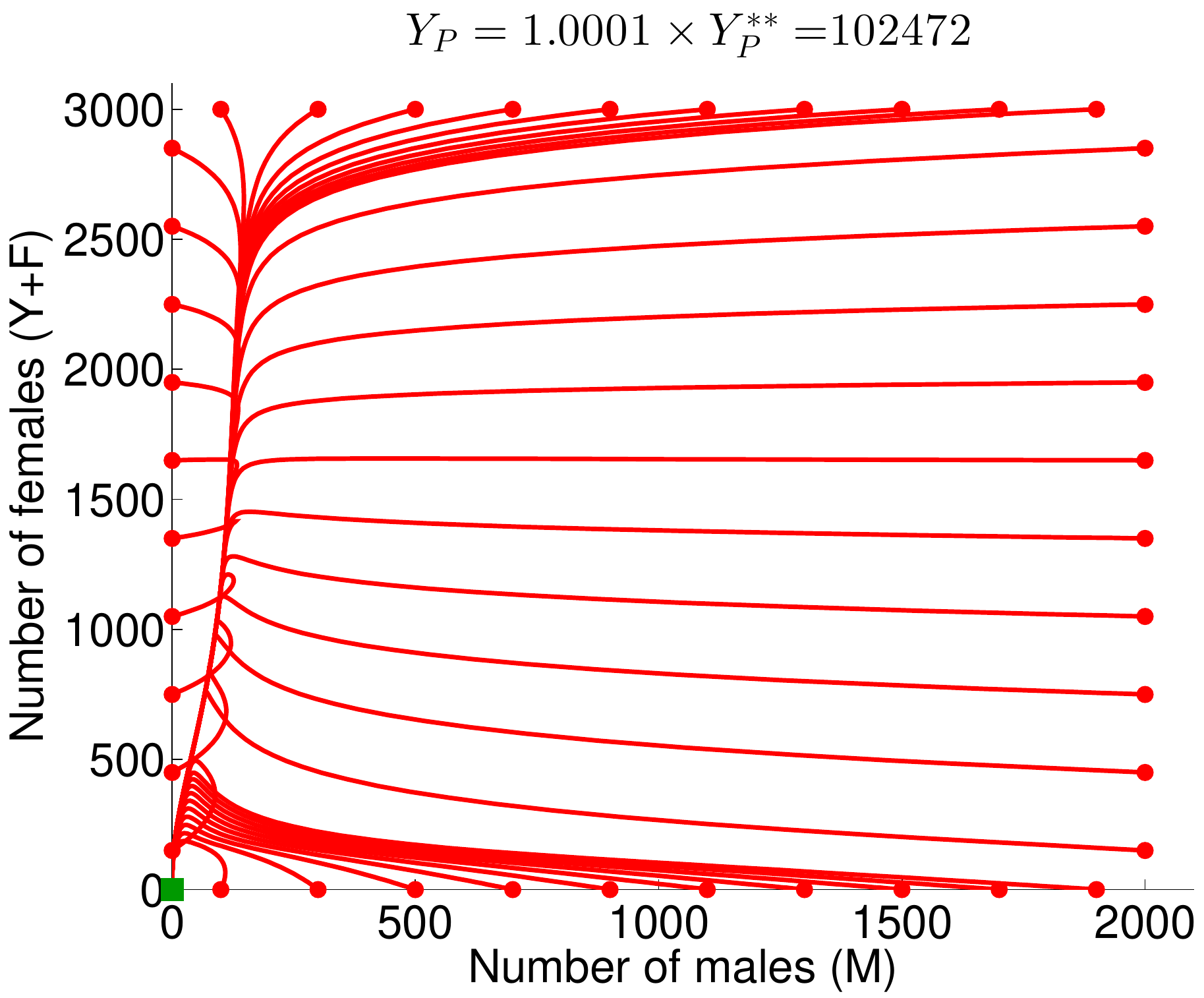}}}
\caption{Trajectories of a set solutions of system (\ref{EqModTemp}) in the $M\times (Y+F)$-plane initiated at the dots. The green squares represent the asymptotically stable equilibria.} 
\label{Fig_Exp3}
\end{center}
\end{figure}

Note that in Theorem~\ref{Thm_GAS_TE}, the GAS of $TE$ is established for $Y_P>\tilde{Y}_P^{**}$, however, we show numerically that the GAS property of $TE$ holds for $Y_P>Y_P^{**}$. From Figure~\ref{Fig_Exp4} we can see that the error between $\tilde{Y}_P^{**}$ and $Y_P^{**}$ is of order $10^4$ which is small compared to the values of $\tilde{Y}_P^{**}$ and $Y_P^{**}$ which are of order $10^6$.
\begin{figure}[H]
\centering
\includegraphics[width=7cm]{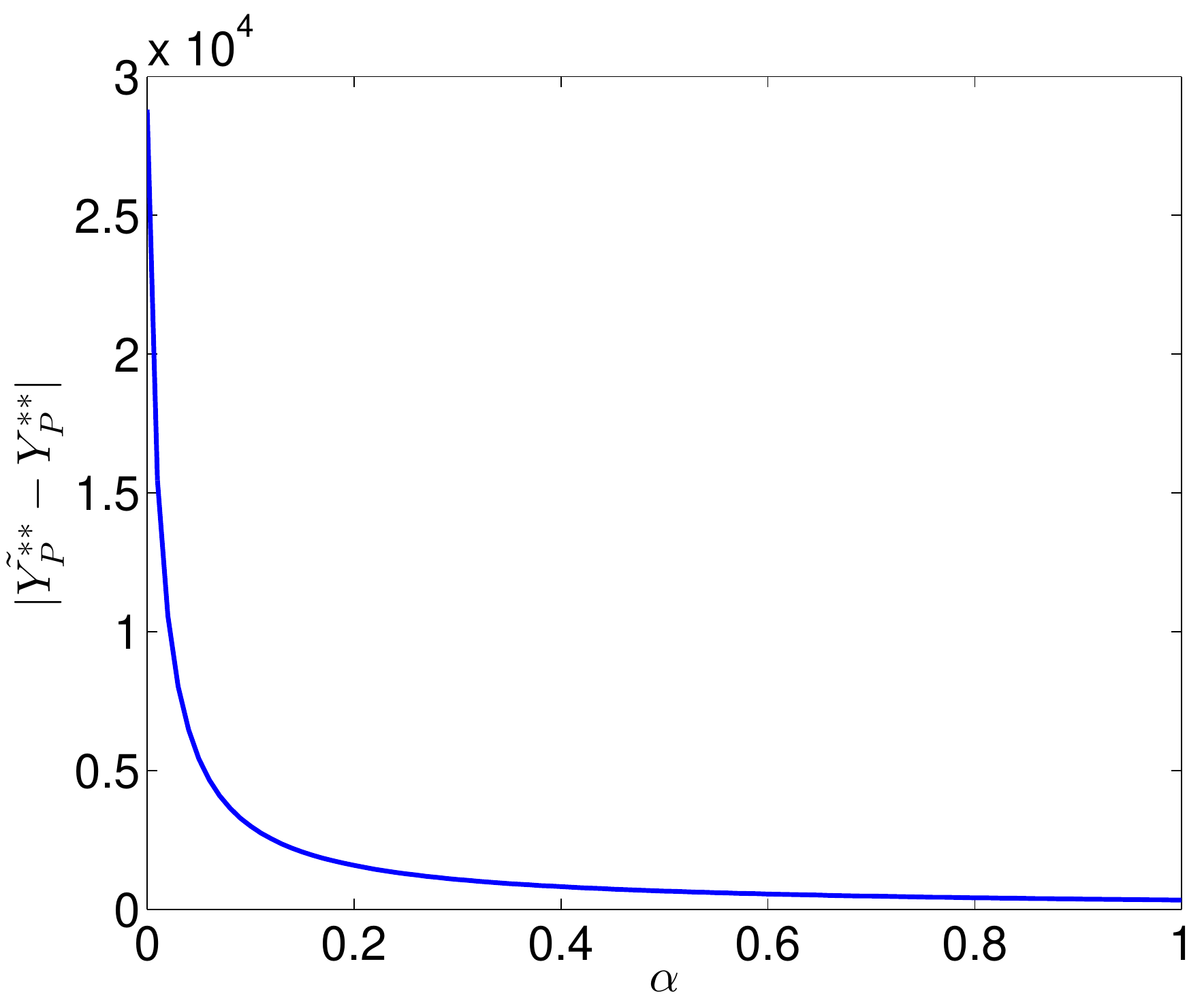}
\caption{Error between the thresholds $\tilde{Y}_P^{**}$ and $Y_P^{**}$ as function of $\alpha$.}
\label{Fig_Exp4}
\end{figure}

% % % % % % % % % % % % % % % % % % % % % % % % % % % % % % % % % % % % % % % % % % % %
% Conclusion
% % % % % % % % % % % % % % % % % % % % % % % % % % % % % % % % % % % % % % % % % % % %
\section*{Conclusion and Perspectives}

Controlling insect pest population in environmentally respectful manner is a main challenge in IPM programs. Mating disruption using female sex-pheromone based lures falls within IPM requirements as it is species specific and leaves no toxic residues in the produce grown. In this work, we build a generic model, governed by a system of ODEs to simulate the dynamics of a pest population and its response to mating disruption control with trapping. From the theoretical analysis of the model, we identify two threshold values of biological interest for the strength of the lure. The first threshold, $Y_P^*$, corresponds to a minimum amount of pheromone necessary for mating disruption to have an effect on the population. However, we show that the control can only be fully efficient, that is, drive an established population to extinction, for an amount of pheromone above a second threshold, $Y_P^{**}$. We also show asymptotic stability of the trivial equilibrium, $TE$, whenever $Y_P>0$. In other words, a small amount of pheromone can be efficient on a very small population, like at invasion stage. Despite the different modelling approach and control method considered, similar results were found by Barclay and Mackauer~\cite{barclay1980sterile}, where a threshold for the size of the release of sterile males was identified, below which two positive equilibria were found, the larger one being stable and the smaller one being unstable, and above which SIT control is fully effective. These theoretical results are consistent with field observation, where the failure of mating disruption is often attributed to a wrong dosage of the pheromone and/or to an excessive population density~\cite{ambrogi2006efficacy,michereff2000initial}.
Further, we show that increasing the capture efficiency of the traps can reduce considerably these threshold values. From practical point of view this suggests that there is an optimal combination of the strength of the pheromone attractant and the capture efficiency of the traps to optimise the control of the pests in terms of population control and cost. These results support the conclusions of Yamanaka~\cite{yamanaka2007mating} stipulating that in the case where the lure used for mating disruption is strong enough, then additional trapping is not necessary, while otherwise, the author advise to rather focus the effort on the trapping efficiency. In a more realistic setting, this optimised control corresponds to an optimal setting of traps releasing the pheromone. Alternative approaches, such as individual based models (IBM), have been considered to study the impact of mating disruption, incorporating a spatial component where the attraction of males is governed by the pheromone plume~\cite{yamanaka2003individual}, or by the effective attraction radius (EAR) which corresponds to the probability of finding the source~\cite{byers2007simulation}. A next step for this work is to add a spatial component to investigate how traps should be set, how many should be used, and how far from each other they should be positioned. Investigation on trap settings and their interactions have been studied in a parallel work dealing with parameter estimation~\cite{dufourd2013parameter} and extended in~\cite{dufourd2015}, where spatio-temporal trapping models are governed by advection-diffusion-reaction processes. 

Another perspective is the validation of the model using field data obtained by Mark-Release-Recapture (MRR) experiments. MRR consists of releasing marked insects in the wild population and recapture them in traps. In such experiments, the number of insects released, as well as the position of the release and the position of the traps are known. Therefore, comparing the trapping data obtained with the model with those obtained using MRR experiments would enable us to validate the model. Then, following the protocol proposed in~\cite{dufourd2015}, unknown parameters of the model could be estimated.

% % % % % % % % % % % % % % % % % % % % % % % % % % % % % % % % % % % % % % % % % % % %
% Acknowledgements
% % % % % % % % % % % % % % % % % % % % % % % % % % % % % % % % % % % % % % % % % % % %
\section*{Acknowledgements}
This work has been supported by the French Ministry of Foreign Affairs and International Development and the South African National Research Foundation in the framework of the PHC PROTEA 2015 call. The support of the DST-NRF Centre of Excellence in Mathematical and Statistical Sciences (CoE-MaSS) towards this research is hereby acknowledged. Opinions expressed and conclusions arrived at, are those of the authors and are not necessarily to be attributed to the CoE.

% % % % % % % % % % % % % % % % % % % % % % % % % % % % % % % % % % % % % % % % % % % %
% Annexes A
% % % % % % % % % % % % % % % % % % % % % % % % % % % % % % % % % % % % % % % % % % % %
\section*{Appendix A: Computation of the basic offspring number}
\label{AppendixR0}
The basic offspring number, sometimes called ``net reproduction rate or ratio''~\cite{birch1948intrinsic}, is defined as the expected number of females originated by a single female in a lifetime~\cite{sallet2010inria}. The computation can be done using a similar method as for the computation of the ``basic reproduction number'' in epidemiological model which determines the number if secondary infections produced by a single infectious individual~\cite{van2002reproduction,diekmann2009construction}.
Let
\begin{itemize}
\item $\mathcal{R}_i(x)$ be the rate of recruitment of new individuals in compartment $i$,
\item $\mathcal{T}_i^+(x)$ be the transfer of individuals into compartment $i$, and
\item $\mathcal{T}_i^-(x)$ be the transfer of individuals out of compartment $i$.
\end{itemize}
Our system can be written in the form:
\begin{equation}
\dot{x}_i=\mathcal{R}_i(x)-\mathcal{T}_i(x), \nonumber
\end{equation}
with
\begin{equation}
\mathcal{T}_i=\mathcal{T}_i^-(x)-\mathcal{T}_i^+(x),\nonumber
\end{equation}
$i=1,..,4$.
When $\gamma M>Y+Y_P$, system (\ref{EqModTemp}) is reduced to system (\ref{EqModTempAbundance}) and it can be written as
\begin{equation}
\dot{x}=\mathcal{R}(x)-\mathcal{T}(x),\nonumber
\end{equation}
with
\begin{equation}
\mathcal{R}(x)=
\left(
\begin{array}{c}
b(1-\frac{I}{K})F\\
0\\0\\0\\
\end{array}
\right),\nonumber
\end{equation}
end
\begin{equation}
\mathcal{T}(x)=
\left(
\begin{array}{c}
(\nu_I+\mu_I)I\\
(\nu_Y+\mu_Y)Y-r\nu_II-\delta F\\
(\delta+\mu_F)F-\nu_Y Y\\
\mu_M M-(1-r)\nu_I I\\
\end{array}
\right).\nonumber
\end{equation}
To obtain the next generation operator, we compute the Jacobian matrices of $\mathcal{R}$ and $\mathcal{T}$, respectively, $J_{\mathcal{R}}$ and $J_{\mathcal{T}}$. Then, from~\cite{heesterbeek2000mathematical,van2002reproduction}, the next generation operator is defined as $RT^{-1}$
where $R=J_{\mathcal{R}}(TE)$ and $T=J_{\mathcal{T}}(TE)$.
Here, we have
\begin{equation}
R=
\left(
\begin{array}{cccc}
0 & 0 & b & 0\\
0 & 0 & 0 & 0\\
0 & 0 & 0 & 0\\
0 & 0 & 0 & 0\\
\end{array}
\right), \textrm{ and }\quad
T=
\left(
\begin{array}{cccc}
\nu_I+\mu_I & 0 & 0 & 0\\
r\nu_I & \nu_Y+\mu_Y & -\delta & 0\\
0 & -\nu_Y & \delta+\mu_F & 0\\
(1-r)\nu_I & 0 & 0 & \mu_M\\
\end{array}
\right).\nonumber
\end{equation}
The basic offspring number is obtained by computing the spectral radius of the next generation operator~\cite{van2002reproduction,diekmann2009construction}:
\begin{equation}
\mathcal{N}_{0}=\rho(RT^{-1})=\frac{br\nu_I \nu_Y}{(\mu_I+\nu_I)\left((\nu_Y+\mu_Y)(\delta+\mu_F)-\delta \nu_Y\right)}.\nonumber
\end{equation}

% % % % % % % % % % % % % % % % % % % % % % % % % % % % % % % % % % % % % % % % % % % %
% Annexes B
% % % % % % % % % % % % % % % % % % % % % % % % % % % % % % % % % % % % % % % % % % % %
\section*{Appendix B: Computation of the endemic equilibrium}
\label{AppendixEE}
We seek for $(I^*, Y^*, F^*, M^*)$ such that
\begin{equation}
\left\lbrace
\begin{array}{lcl}
\frac{dI^*}{dt} & = & 0,\\
\frac{dY^*}{dt} & = & 0,\\
\frac{dF^*}{dt} & = & 0,\\
\frac{dM^*}{dt} & = & 0,
\end{array}
\right.\nonumber
\end{equation}
that is
\begin{equation}
\left\lbrace
\begin{array}{lcl}
 b\left(1-\frac{I^*}{K}\right)F^*-(\nu_I+\mu_I)I^*& = & 0,\\
 r\nu_I I^*-(\nu_Y+\mu_Y)Y^*+\delta F^*& = & 0,\\
 v_Y Y^* -(\delta+\mu_F)F^*& = & 0,\\
 (1-r)\nu_I I^* -\mu_M M^*& = & 0.
\end{array}
\right.\label{AppendEq}
\end{equation}
Thus, starting from the bottom:
\begin{eqnarray}
 & & M^* =  \frac{(1-r)\nu_I}{\mu_M}I^*,\nonumber
 \end{eqnarray}
\begin{eqnarray}
 & & F^* =  \frac{\nu_Y Y^*}{(\delta+\mu_F)},\nonumber
 \end{eqnarray}
 \begin{eqnarray}
 & & r\nu_I I^*-(\nu_Y+\mu_Y)Y^*+ \frac{\delta \nu_Y Y^*}{(\delta+\mu_F)} =  0 ,\nonumber\\ \Leftrightarrow
  & & Y^*=\frac{r\nu_I(\delta+\mu_F)}{(v_Y+\mu_Y)(\delta+\mu_F)-\delta \nu_Y}I*, \nonumber
  \end{eqnarray}
  and therefore,
\begin{eqnarray}
 & & F^* =  \frac{r\nu_I\nu_Y}{(v_Y+\mu_Y)(\delta+\mu_F)-\delta \nu_Y}I*. \nonumber
 \end{eqnarray}
Using the first equation of (\ref{AppendEq}),
  \begin{eqnarray}
  & & b\left(1-\frac{I^*}{K}\right)F^*-(\nu_I+\mu_I)I^*=0 \nonumber\\ \Leftrightarrow
  & & b\left(1-\frac{I^*}{K}\right)\frac{\nu_Y Y^*}{(\delta+\mu_F)}-(\nu_I+\mu_I)I^*=0 \nonumber\\ \Leftrightarrow
    & & b\left(1-\frac{I^*}{K}\right)\frac{\nu_Y}{(\delta+\mu_F)}\frac{r\nu_I(\delta+\mu_F)}{(v_Y+\mu_Y)(\delta+\mu_F)-\delta \nu_Y}I^*-(\nu_I+\mu_I)I^*=0 \nonumber\\ \Leftrightarrow
    & & b\left(1-\frac{I^*}{K}\right)\frac{r\nu_I\nu_Y}{(v_Y+\mu_Y)(\delta+\mu_F)-\delta \nu_Y}-(\nu_I+\mu_I)=0 \nonumber\\\Leftrightarrow
    & & b\left(1-\frac{I^*}{K}\right)\frac{r\nu_I\nu_Y}{(v_Y+\mu_Y)(\delta+\mu_F)-\delta \nu_Y}-\frac{(\nu_I+\mu_I)((v_Y+\mu_Y)(\delta+\mu_F)-\delta \nu_Y)}{(v_Y+\mu_Y)(\delta+\mu_F)-\delta \nu_Y}=0 \nonumber\\ \Leftrightarrow
    & & b\left(1-\frac{I^*}{K}\right)r\nu_I\nu_Y-(\nu_I+\mu_I)((v_Y+\mu_Y)(\delta+\mu_F)-\delta \nu_Y)=0 \nonumber\\ \Leftrightarrow
     & & br\nu_I\nu_Y I^*=(-(\nu_I+\mu_I)((v_Y+\mu_Y)(\delta+\mu_F)-\delta \nu_Y)+br\nu_I\nu_Y)K \nonumber\\ \Leftrightarrow
     & & I^*=\frac{-(\nu_I+\mu_I)((v_Y+\mu_Y)(\delta+\mu_F)-\delta \nu_Y)+br\nu_I\nu_Y}{b\nu_I\nu_Y}K \nonumber\\ \Leftrightarrow
     & & I^*=\frac{-(\nu_I+\mu_I)((v_Y+\mu_Y)(\delta+\mu_F)-\delta \nu_Y)}{b\nu_I\nu_Y}K+K \nonumber\\ \Leftrightarrow
     & & I^*=\left(1-\frac{1}{\mathcal{N}_{\delta}}\right)K. \nonumber
\end{eqnarray}

\section*{Appendix C: Preliminaries on monotone dynamical systems}

Consider the autonomous system of ODEs
\begin{equation}
\frac{dx}{dt}=f(x)
\label{Sys_Prelim}
\end{equation}
where $f:D \longrightarrow \mathbb{R}^n$, $D\subset \mathbb{R}^n$. We assume that $f$ is locally Lipshitz so that local existence and uniqueness of the solution is assured.
We will use the following notations. Denote 
\[x(x_0,t)\]
the solution of (\ref{Sys_Prelim}) initiated in $x_0$.
Further, for $x, y\in\mathbb{R}_+^n$, we have:
$$x\leq y \Longleftrightarrow x_i\leq y_i, \forall i\in \{1,2,\dots,n\},$$
%\begin{itemize}
%\item $x\leq y \Longrightarrow x_i\leq y_i, \forall i\in \{1,2,\dots,n\},$
%\item $x < y \Longrightarrow x\leq y, \textrm{ and } x_i< y_i \textrm{ for some } i\in \{1,2,\dots,n\},$
%\item $x \ll y \Longrightarrow x_i < y_i, \forall i\in \{1,2,\dots,n\}.$
%\end{itemize}

\begin{definition}
System (\ref{Sys_Prelim}) is said to be cooperative if for every $i=1,...,n$ the function $f_i(x)$ is monotone increasing with respect to all $x_j$, $j=1,...,n$, $j\neq i$.
\end{definition}

\begin{theorem} \label{TheoQuasiMon}
If $f$ is differentiable on $D$ then the system (\ref{Sys_Prelim}) is cooperative if and only if
\begin{equation}
\frac{\partial f_i}{\partial x_j}(x)\geq 0, \quad i\neq j, \quad x\in D.\nonumber
\end{equation}
\end{theorem}

\begin{theorem}
Let (\ref{Sys_Prelim}) be a cooperative system and let $x(x_0,t)$ be a solution of (\ref{Sys_Prelim}) on $[0,T)$. If $y(t)$ is a differentiable function on $[0,T)$ satisfying
\begin{equation}
\frac{d y}{d t}\leq f(y), \quad y(0)\leq x_0,\nonumber
\end{equation}
then
\begin{equation}
y(t)\leq x(x_0,t), \quad t\in[0,T). \nonumber
\end{equation}
\cite[Theorem II, 12, II]{walter2012differential}
\label{Thm_Monotonicity}
\end{theorem}

\begin{theorem}
Let (\ref{Sys_Prelim}) be a cooperative system, and $a,b\in D$. If $a\leq b$ and if for $t>0$, $x(a,t)$ and $x(b,t)$ are defined, then $x(a,t)\leq x(b,t)$.
\end{theorem}
(\cite[Prop. 1.1 p32]{smith2008monotone})

\begin{theorem}
Let $a,b\in D$, such that $a\leq b, [a,b] \subseteq D$ and $f(a) \leq \textbf{0} \leq f(b)$. Then (\ref{Sys_Prelim}) defines a positive dynamical system on $[a,b]$. Moreover, if $[a,b]$ contains a unique equilibrium $p$ then, $p$ is globally asymptotically stable (GAS) on $[a,b]$.
\label{Thm_GAS_MonSys}
\end{theorem}
(\cite[Thm. 3.1 p18]{smith2008monotone}, \cite[Thm. 6]{anguelov2012})

\begin{theorem}
Let  $a,b\in D$, such that $a\leq b$, $[a,b]\subseteq D$ and $f(a)=f(b)=0$ for (\ref{Sys_Prelim}). Then
\begin{itemize}
\item[a)] (\ref{Sys_Prelim}) defines a positive dynamical system on $[a,b]$.
\item[b)] If $a$ and $b$ are the only equilibria of the dynamical system on $[a,b]$, then all non-equilibrium solutions initiated in $[a,b]$ converge to one of them, that is, either all converge to $a$ or all converge to $b$.
\end{itemize}
\label{Thm_BasinAttraction}
\end{theorem}

\begin{proof}
a) The proof follows from Theorem~\ref{Thm_GAS_MonSys}.

b) Assume that there exist in the interior of the order interval $[a,b]$ two points $x_1$ and $x_2$ such that $x(x_1,t)$...
\cite[Thm. 2.3.2, Prop. 2.2.1]{smith2008monotone}
\end{proof}

% % % % % % % % % % % % % % % % % % % % % % % % % % % % % % % % % % % % % % % % % % % %
% References
% % % % % % % % % % % % % % % % % % % % % % % % % % % % % % % % % % % % % % % % % % % %
%\section*{References}

%\bibliographystyle{model1b-num-names}
\bibliographystyle{plain}
\bibliography{Biblio}

\end{document}